\documentclass[11pt,a4paper]{article}

\usepackage[T1]{fontenc}
\usepackage{epsfig}
\usepackage{amsmath, amsthm}
\usepackage{amsfonts,amssymb}
\usepackage{derivative}
\usepackage{mathtools}
\usepackage[applemac]{inputenc}		
\usepackage{hyperref}
\usepackage{comment}

\hypersetup{colorlinks=true,urlcolor=black, pdftitle="WBMTII"}

\newtheorem{conj}{Conjecture}[section]
\newtheorem{thm}[conj]{Theorem}

\newtheorem{rem}[conj]{Remark}
\newtheorem{lem}[conj]{Lemma}
\newtheorem{prop}[conj]{Proposition}

\newtheorem{ques}[conj]{Question}

\newtheorem{defn}[conj]{Definition}
\newtheorem{cor}[conj]{Corollary}

\makeatletter
\newtheorem*{rep@theorem}{\rep@title}
\newcommand{\newreptheorem}[2]{%
\newenvironment{rep#1}[1]{%
 \def\rep@title{#2 \ref{##1}}%
 \begin{rep@theorem}}%
 {\end{rep@theorem}}}
\makeatother
\newreptheorem{theorem}{Theorem}
\newreptheorem{corollary}{Corollary}
\newreptheorem{proposition}{Proposition}

\newcommand{\vol}{\mathrm{Vol}}

\newcommand{\R}{\mathbb{R}}
\newcommand{\N}{\mathbb{N}}

 \usepackage[backend=biber,style=numeric-comp,doi=false, bibencoding=utf8,giveninits=true, maxnames=50, isbn=false,url=false]{biblatex}
 \DeclareNameAlias{default}{family-given}
\def\s{\mathbb{S}}

\def\R{{\mathbb R}}

\def\phi{\varphi}

\def\bee{\begin{eqnarray*}}
\def\ene{\end{eqnarray*}}

\newcommand\nnfootnote[1]{%
  \begin{NoHyper}
  \renewcommand\thefootnote{}\footnote{#1}%
  \addtocounter{footnote}{-1}%
  \end{NoHyper}
}

\begin{document}

\title{Weighted Brunn-Minkowski Theory II \\
    \large Inequalities for Mixed Measures and Applications}

\author{Matthieu Fradelizi\thanks{Supported in part by the  B\'ezout Labex funded by ANR, reference ANR-10-LABX-58}, Dylan Langharst\footnotemark[1] 
 \thanks{Supported in part by the U.S. National Science Foundation Grant DMS-1101636 and
the United States - Israel Binational Science Foundation (BSF) Grant 2018115} 
\thanks{Funded by the FSMP Post-doctoral program},
Mokshay Madiman, 
and Artem Zvavitch\footnotemark[1] \footnotemark[2]}

\date{\today}
\maketitle

\begin{abstract}
In ``Weighted Brunn-Minkowski Theory I", the prequel to this work, we discussed how recent developments on concavity of measures have laid the foundations of a nascent weighted Brunn-Minkowski theory. In particular,
 we defined the mixed measures of three convex bodies and obtained its integral representation. In this work, we obtain inequalities for mixed measures, such as a generalization of Fenchel's inequality; this provides a new, simpler proof of the classical volume case. Moreover, we show that mixed measures are connected to the study of log-submodularity and supermodularity of the measure of Minkowski sums of convex bodies. This elaborates on the recent investigations of these properties for the Lebesgue measure. We conclude by establishing that the only Radon measures that are supermodular over the class of compact, convex sets are multiples of the Lebesgue measure. Motivated by this result, we then discuss weaker forms of supermodularity by restricting the class of convex sets.
\end{abstract}
\tableofcontents

\section{Introduction}
\nnfootnote{Keywords: Brunn-Minkowski theory, surface area, Gaussian measure, zonoids, mixed volumes, mixed measures

Mathematics Subject Classification 2020 - Primary: 52A20 and 52A21, Secondary: 46T12
}
The study of Borel measures on $\R^n$ having concavity has a rich history. We recall for compact sets $K,L\subset\R^n,$ their \textit{Minkowski sum} is precisely $K+L=\{a+b:a\in K,b\in L\}.$ Recall that a Borel measure is said to be Radon if it is locally finite and regular. We say a collection $\mathcal{C}$ of Borel subsets of $\R^n$ is a \textit{class} if it is closed under Minkowski summation and dilation, i.e. $K,L\in\mathcal{C}$ implies $tK+sL\in\mathcal{C}$ for all $t,s\geq 0$. The study of concavity of measures can be encapsulated in the idea of $F$-concave measures:
A Radon measure $\mu$ on $\R^n$ is said to be $F$-concave on a class $\mathcal{C}$ if there exists a continuous, strictly monotonic function $F:(0,\mu(\R^n))\to (-\infty,\infty)$ such that, for every $t \in [0,1]$ and for every pair $K,L \in \mathcal{C}$  such that $\mu(K),\mu(L)>0$ one has
  \begin{equation} \mu((1-t)K +t L)\geq F^{-1}\left((1-t)F(\mu(K)) +t F(\mu(L))\right). \label{eq:fcon} \end{equation}
  If $F(x)=x^s$, $s\in \R\setminus\{0\}$, then we say that the measure $\mu$ is $s$-concave. For $F(x)=\ln(x)$, the measure $\mu$ is $\log$-concave or $0$-concave and \eqref{eq:fcon} takes the form 
  $$\mu((1-t) K +tL)\geq\mu(K)^{1-t}\mu(L)^{t}.$$
  Recall that a Borel measure is said to have \textit{density} if it has a locally integral Radon-Nikodym derivative with respect to the Lebesgue measure.  C. Borell's classification of $s$-concave Radon measures \cite[Theorem 3.2]{Bor75a} shows they have a density with respect to the Lebesgue on the appropriate subspace that also has a concavity property.  Specifically, $\vol_n$, the Lebesgue measure on $\R^n,$ is $1/n$-concave on the class of all compact subsets of $\R^n$; this is classically known as the \textit{Brunn-Minkowski inequality}. 
An important $\log$-concave measure is the standard Gaussian measure on $\R^n,$ which is given by $$d\gamma_n(x)=\frac{1}{(2\pi)^{n/2}}e^{-|x|^2/2}dx,$$
where $|\cdot|$ denotes the standard Euclidean norm. However, the Gaussian measure actually has other types of concavity than just log-concavity. Indeed, the \textit{Ehrhard inequality} states that $\gamma_n$ is $\Phi^{-1}$-concave on the class of all Borel subsets of $\R^n$, where $\Phi(x)=\gamma_1((-\infty,x])$. It was first proven by A. Ehrhard for the case of two closed, convex sets \cite{Ehr83}.  R. Lata{\l}a \cite{Lat96} generalized Ehrhard's result to the case of one arbitrary Borel set and one convex set; the general case for two Borel sets was proven by C. Borell \cite{Bor03}. Since $\Phi$ is log-concave, the log-concavity of the Gaussian measure is weaker than the Ehrhard inequality.

In recent decades, many people have worked to move beyond Borell's classification by restricting the class $\mathcal{C}$; usually, $\mathcal{C}$ is taken to be a class of convex bodies (compact, convex sets in $\R^n$ with non-empty interior). A set $K$ is said to be symmetric if $K=-K.$ R. Gardner and the last named author \cite{GZ10} conjectured that, for symmetric convex bodies $K$ and $L$ and $t\in[0,1]$,
\begin{equation}\label{e:gamma_gaussian}
    \gamma_n\left((1-t) K + t L\right)^{1/n}\geq (1-t)\gamma_n(K)^{1/n} + t \gamma_n(L)^{1/n},
\end{equation}
i.e. $\gamma_n$ is $1/n$-concave over the class of symmetric convex bodies.  An example given in \cite{NT13:1} shows that assumption on $K$ and $L$ having some symmetry is necessary. An important progress was made in \cite{KL21} by A. Kolesnikov and G. Livshyts; they showed that the Gaussian measure is $\frac{1}{2n}$-concave on the class of convex bodies containing the origin in their interior \cite{KL21}. Inequality \eqref{e:gamma_gaussian} was finally proven by A. Eskenazis and G. Moschidis in \cite{EM20:3} for symmetric convex bodies. 
Let $\mathcal{M}_n$ be a  class of Borel measures $\mu$ on $\R^n$ with the following properties: 
\begin{equation}
\label{eq:CER}
d\mu=e^{-W(|x|)}, \;\;  W:(0,\infty)\to(-\infty,\infty] \mbox{ is increasing  and } t\mapsto W(e^t) \text{ is convex}.\end{equation}
This class contains every rotational invariant, log-concave measure. D. Cordero-Erausquin and L. Rotem \cite{CR23} extended the result by Eskenazis and Moschidis to every measure $\mu\in\mathcal{M}_n,$ i.e. every Borel measure $\mu\in\mathcal{M}_n$ is $1/n$-concave over the class of symmetric convex bodies.

The study of the interaction of the volume of convex bodies and their Minkowski sums is known as the \textit{Brunn-Minkowski theory}; the Brunn-Minkowski theory is thoroughly detailed in the textbook of R. Schneider \cite{Sch14:book}, and we will make frequent reference to it. Following foundational ideas by G. Livshyts \cite{Liv19} and  E. Milman and L. Rotem \cite{MR14}, we explored in Part I of this work \cite{FLMZ24:1} extensions of other concepts from the Brunn-Minkowski theory to measures, in what we call the \textit{weighted Brunn-Minkowski Theory}. Primarily, we had discussed surface area, mixed surface area, and the \textit{mixed volume} of convex bodies. 

The following facts can be found in \cite{Sch14:book}. Steiner's formula states that the Minkowski sum of the dilates of two compact, convex sets can be expanded as a polynomial of degree $n$: for every $t \geq 0$, convex body $K$ and compact, convex set $L$, both in $\R^n,$ one has
$$\vol_n(K+tL)=\sum_{j=0}^n \binom{n}{j} V(K[n-j],L[j])t^j,$$
with non-negative coefficients $V(K[n-j],L[j])$ called the \textit{mixed volumes} of $(n-j)$ copies of $K$ and $j$ copies of $L$. For $j=1$, one writes $V(K[n-1],L)$. Taking the derivative, one obtains
\begin{equation}
	    V(K[n-1],L)=\frac{1}{n}\lim_{\epsilon\to0^+}\frac{\vol_n(K+\epsilon L)-\vol_n(K)}{\epsilon}.
	    \label{eq:mixed_0}
\end{equation}
The first step in a weighted Brunn-Minkowski theory is to generalize mixed volumes. Since \eqref{eq:mixed_0} has nothing to do with the concavity of the volume, when given an arbitrary Borel measure $\mu,$ and Borel sets $K$ and $L$, the $\mu$-mixed measure of $K$ and $L$ can be defined as
	\begin{equation}
\mu(K;L):=\liminf_{\epsilon\to0^+}\frac{\mu(K+\epsilon L)-\mu(K)}{\epsilon},
	    \label{eq:arb_mixed_0}
	\end{equation}
when the $\liminf$ is finite. Heuristically, if the limit exists, this is precisely the first coefficient in the Taylor series expansion of $\mu(K+tL)$ (in the variable $t$). This terminology was introduced by Livshyts in \cite{Liv19}, and has been used in other works recently, see e.g. \cite{Liv19,Hos21,LRZ22,KL23,FLMZ24:1}. It has appeared previously in many works without being explicitly given the name mixed measures, see e.g. \cite{Naz03,MR14,KM18:1,CNV04,ST74}. For Borel sets $K$ and $L$ containing the origin with finite $\mu$ measure, the limit exists when $\mu$ has continuous density. If $\lambda_n$ denotes the Lebesgue measure, then \eqref{eq:arb_mixed_0} is consistent with mixed volumes up to a factor $n$ i.e. $\lambda_n(K; L)=nV(K[n-1],L)$, thus $\lambda_n(K; K)=n\vol_n(K)$.  In general $\mu(K;K)\neq n\mu(K)$, unlike in the volume case. In fact, if $L$ does not contain the origin, then $\mu(K;L)$ could be negative.

Here and henceforth, in an abuse of notation, when the derivative operations $(d/dt), (\partial/\partial s)$ and $(\partial^2/\partial s\partial t)$ are applied to quantities of the form $\mu(tK)$, $\mu(tK;L)$ and $\mu(A+sB+tC)$, we mean one-sided derivatives from above. By iterating the Steiner polynomial, $\vol_n(K+t_1L_1+t_2L_2)$ also has a polynomial expansion in the variables $t_1$ and $t_2$; the coefficient of $t_1t_2$ in the corresponding polynomial expansion, the mixed volume of $(n-2)$-copies of $K,$ one copy of $L_1,$ and one copy $L_2$, is thus precisely
\begin{equation}
    V(K[n-2],L_1,L_2)=\frac{1}{n(n-1)}\pdv{}{t_1,t_2}\vol_n(K+t_1L_1+t_2L_2)(0,0).
    \label{eq:mixed_volume}
\end{equation}
Notice that $\vol_n(K)=V(K[n-1],K)=V(K[n-2],K[2])$. We introduced in \cite{FLMZ24:1} the weighted analogue of \eqref{eq:mixed_volume}. \begin{defn}
\label{def:mixed}
Let $\mu$ be a Borel measure on $\R^n$. Then, for Borel sets $A,B,C\subset\R^n$ with finite $\mu$-measure, the mixed measure of $(n-2)$ copies of $A$, one copy of $B$ and one copy of $C$ is given by
$$\mu(A;B,C)=\pdv{}{s,t}\mu(A+sB+tC)(0,0),$$
whenever the mixed derivative exists.
\end{defn}
From the Schwarz theorem of multivariate calculus, one has $\mu(A;B,C)=\mu(A;C,B)$, which still applies for mixed partial derivatives that are one-sided. In \cite{FLMZ24:1}, briefly discussed in Appendix A, we recall a formula for $\mu(K;L)$ when $K$ and $L$ are convex bodies and $\mu$ has density. We also showed that $\mu(A;B,C)$ exists if $\mu$ has $C^2$ density, and we obtained an explicit formula when $A$ is smooth (i.e. has $C^2_+$ boundary).

The Brunn-Minkowski inequality implies the following Minkowski's first and second inequalities for the mixed volumes: for compact, convex sets $K$ and $L$ in $\R^n$,
\begin{equation}
    V(K [n-1], L)^{n} \geq \vol_n(K)^{n-1} \vol_n(L),
    \label{eq:first_in}
\end{equation}
and
\begin{equation}
\label{eq:second_in}
V(K[n-1], L)^{2} \geq \vol_n(K) V(K[n-2], L [2]).
\end{equation}
Equality holds in inequality \eqref{eq:first_in} if and only if $K$ and $L$ are homothetic.  Minkowski's second inequality is merely a special case of \textit{Minkowski's quadratic inequality}:  for $A,B,C$ compact, convex sets, one has
\begin{equation}\label{eq:quad_in}V(A[n-2],B,C)^2 - V(A[n-2],B,B)V(A[n-2],C,C) \geq 0.\end{equation}
This, in turn, is a special case of the Aleksandrov-Fenchel inequality (which we do not discuss). We do mention however that despite having existed in the literature for over a century, the full equality conditions to Minkowski's second and quadratic inequalities were only established very recently by R. van Handel and Y. Shenfeld \cite{SH21}.

The concavity of the function in two variables $f(s,t)=\vol_n(A+sB+tC)^{1/n}$ implies the following  ``reverse" Minkowski's quadratic inequality, which originally appeared in the works by W. Fenchel (see \cite{Fen36} and also \cite{Sch14:book}) and further generalized in \cite{FGM03, AFO14, SZ16}:
\begin{align}
&\bigg(-V^2(A[n-1],C)\frac{V(A[n-2],B[2])}{\vol_n(A)}
+  V(A[n-1],B)V(A[n-1],C)\frac{2V(A[n-2],B,C)}{\vol_n(A)} 
\nonumber\\
&\quad\quad\quad\quad\quad\quad- V^2(A[n-1],B)\frac{V(A[n-2],C[2])}{\vol_n(A)}\bigg)
\nonumber\\
&\quad\geq V^2(A[n-2],B,C)-V(A[n-2],B[2])V(A[n-2],C[2]).\label{eq:reverse}
\end{align}
Observe that if one uses Minkowski's quadratic inequality \eqref{eq:quad_in} to remove the $V(A[n-1],B)^2$ and $V(A[n-1],C)^2$ terms in the reverse Minkowski's quadratic inequality \eqref{eq:reverse}, one obtains
\begin{align*}2\frac{V(A[n-2],B,C)V(A[n-1],B)V(A[n-1],C)}{\vol_n(A)} &\geq V^2(A[n-2],B,C)
\\
&+V(A[n-2],B[2])V(A[n-2],C[2]).
\end{align*}
Moreover, by throwing away the extra term $V(A[n-2],B[2])V(A[n-2],C[2]),$ which is non-negative, and re-arranging, we obtain \textit{Fenchel's inequality}:
\begin{equation}2V(A[n-1],B)V(A[n-1],C) \geq \vol_n(A)V(A[n-2],B,C).
\label{eq:fen}
\end{equation}

The main result of Section~\ref{sec:mixed} is Theorem~\ref{t:fen}, which generalizes Fenchel's inequality to the case of $s$-concave measures. When the measure is set to be Lebesgue, one recovers the classical volume case (see the discussion before \eqref{eq:sup_fen}). Moreover, our approach yields a proof of \eqref{eq:fen} \textit{without} using Minkowski's quadratic inequality. 

Our first set of results is generalizing inequalities \eqref{eq:first_in}, \eqref{eq:second_in} and \eqref{eq:reverse} to the setting of measures with density and some type of concavity; this is achieved in Proposition~\ref{prop:first_second} and Proposition~\ref{prop:second} for Minkowski's first and second inequality respectively, and in Theorem~\ref{t:min_quad} for the reverse Minkowski's quadratic inequality. For Minkowski's first inequality, the case of convex bodies and a measure $\mu$ with differentiable density was first done in \cite{FLMZ24:1,Liv19}. As for Minkowski's second inequality for $F$-concave measures, this was established previously in the prequel, \cite{FLMZ24:1}. The version that we obtain here is distinct; of course both versions coincide in the case of the Lebesgue measure. We present here the case of the Gaussian measure using \eqref{e:gamma_gaussian}; the reader can deduce from the result in Proposition~\ref{prop:second} other such inequalities for the Gaussian measure and other measures with concavity.

\begin{cor}
    Let $K$ and $L$ be symmetric convex bodies in $\R^n, n \geq 2$. Then:
    $$\gamma_n(K;L)^2\geq\frac{n}{n-1} \gamma_n(K)\gamma_n(K;L,L).$$
\end{cor}
In Theorem~\ref{t:min_quad}, we obtain an analogue of Fenchel's reverse Minkowski's quadratic inequality for $F$-concave  measures, which, for Gaussian measure, using \eqref{e:gamma_gaussian}, gives the following.

\begin{cor}
    Let $A,B$ and $C$ be symmetric convex bodies $\R^n.$ Then:
    \begin{align*}
\gamma_n(A; B,B)&\gamma_n(A;C,C)-\left(\frac{n-1}{n}\right)\gamma_n(A)\left(\gamma_n^2(A;B)\gamma_n(A;C,C)+\gamma_n^2(A;C)\gamma_n(A; B,B)\right) 
\\
&\geq \gamma_n^2(A;B,C)-2\left(\frac{n-1}{n}\right)\gamma_n(A)\gamma_n(A;C)\gamma_n(A;B)\gamma_n(A;B,C).
\end{align*}
\end{cor}
In Section~\ref{sec:logsup}, we show that the study of certain properties of mixed measures of three bodies (which are local properties) is related to specific inequalities concerning measures of Minkowski sums of sets (which are global properties). 
\begin{defn}
    \label{def:supmod_bor}
    Let $\mu$ be a Borel measure on $\R^n$, and let $\mathcal{C}$ be a class of compact subsets of $\R^n$. We say $\mu$ is supermodular over $\mathcal{C}$ if, for any $A,B,C\in\mathcal{C}$, one has
    \begin{equation}
        \mu(A+B+C) +\mu(A) \geq \mu(A+B) + \mu(A+C).
        \label{eq:sup_mod}
    \end{equation}
    If the reverse inequality holds, then $\mu$ is said to be submodular over $\mathcal{C}$.
\end{defn}

We note in passing that there also exists a notion of supermodularity for functions as opposed to measures, which has arisen in the the optimization, probability, and convex geometry literature; see, e.g., \cite{Fuj05:book, FKG71, MS00, FMZ24, MMR24}. It follows by definition that the function $(a_1,\ldots, a_r)=a\mapsto\mu(\sum_{i=1}^r a_i K_i)$ being supermodular implies that the measure $\mu$ is supermodular with respect to Minkowski summation. 

The examples of classes $\mathcal{C}$ that we consider include all compact, convex sets, convex bodies, convex bodies containing the origin, symmetric convex bodies and zonoids (see after \eqref{eq:zon} for definition of zonoids). In \cite{FMMZ18}, a subset of the authors together with A. Marsiglietti showed that the Lebesgue measure is supermodular over the class of $n$-dimensional convex bodies. That is, for convex bodies $A,B,$ and $C,$ in $\R^n,$ one has
$$\vol_n(A)+\vol_n(A+B+C)\geq \vol_n(A+B)+\vol_n(A+C).$$

We start Section~\ref{sec:lebclass} by localizing the property of supermodularity in the next proposition.
\begin{prop}[Local form of supermodularity]
\label{p:sup_deriv_equiv_0}
    Let $\mu$ be a Radon measure on $\R^n$. Let $\mathcal{C}$ be a class of convex sets such that the limits in the definitions of $\mu(A;B)$ and $\mu(A;B,C)$ exist for every $A,B,C\in\mathcal{C}$. The following are equivalent:
    \begin{enumerate}
        \item $\mu$ is supermodular over $\mathcal{C}$.
        \item  $\mu(A+C;B)\geq \mu(A;B)$ for every $A,B,C\in\mathcal{C}$.
        \item $\mu(A;B,C) \geq 0$ for every $A,B,C\in\mathcal{C}$.
    \end{enumerate}
\end{prop}

The main result of Section~\ref{sec:lebclass} is the following theorem, which shows that constant multiples of the Lebesgue measure are the only Radon measures supermodular over the class of compact, convex sets.
\begin{reptheorem}{t:Radon}
    Let $n\geq 1$. Let $\mu$ be a Radon measure that is supermodular (i.e. satisfies \eqref{eq:sup_mod}) over the class of all convex bodies in $\R^n$. Then $\mu$ is a multiple of the Lebesgue measure.
\end{reptheorem}
\noindent In the proof of Theorem~\ref{t:Radon}, we use  (2.) from Proposition~\ref{p:sup_deriv_equiv_0} and work at the level of mixed measures of two bodies. However, this theorem and (3.) from Proposition~\ref{p:sup_deriv_equiv_0} implies the following result: \textit{let $\mu$ be a Radon measure that is not a multiple of the Lebesgue measure, such that $\mu(A;B,C)$ exists for every compact, convex sets $A,B$ and $C$. Then, there exists $A,B$ and $C$ such that $\mu(A;B,C) <0$.} This elaborates on a result shown in \cite{FLMZ24:1}. Indeed, we had shown that for every $R>0$ and $\theta\in\s^{n-1}$ one has
$$\gamma_n(RB_2^n;[-\theta,\theta],[-\theta,\theta]) < 0$$
and for every compact, convex $B,C\subset \R^2,$ there exists $R>0$ such that
$$\gamma_2(RB_2^2;B,C) <0.$$
Here, $B_2^n$ is the unit Euclidean ball in $\R^n$. We also set for later $\s^{n-1}=\partial B_2^{n}$, the unit sphere. Under this notation, we have that $\mu^+(\partial K):=\mu(K;B_2^n)$, where $\mu^+$ is the Minkowski content of $K$ with respect to $\mu$ defined in \eqref{eq_bd} below.

 It was proved by G. Saracco and G. Stefani in  \cite{SS24:1} that if
     $\mu$ is a Borel measure on $\R^n$ {\em with locally integrable density} and  for any $K$ and $L$ convex sets such that $K\subset L$ we have
    \[
    \mu^+(\partial K) \leq \mu^+(\partial L),
    \]
    then, $\mu$ is a multiple of the Lebesgue measure.
    
    Motivated by Theorem~\ref{t:Radon} and the result by Saracco and Stefani, we consider in Section~\ref{sec:weaksup} the question of characterizing the Lebesque measure via inequalities for the surface area of the Minkowski sum. We start with the following theorem:
\begin{reptheorem}{t:surarecom}
For $n\geq 1$, let $\mu$ be a Radon measure on $\R^n$ such that, for every convex body $K$ and compact, convex set $L$, 
    \begin{equation}\mu^+(\partial (K+L)) \geq \mu^+(\partial K).
    \label{eq:sursum}
    \end{equation}
    Then $\mu$ is a multiple of the Lebesgue measure.
\end{reptheorem}
Let us remark that the convex bodies $K$ and $L$ constructed by Saracco and Stefani, whose structure forced the density of $\mu$ to be constant, were not of the forms $L=K+C$, where $C$ is a compact, convex set; thus, the result in \cite{SS24:1} does not imply Theorem~\ref{t:surarecom}. We remark also that as a consequence of Proposition~\ref{p:sup_deriv_equiv_0}, \eqref{eq:sursum} is equivalent to supermodularity in the case when the compact, convex set $B$ in \eqref{eq:sup_mod} is taken from the class of dilates of the Euclidean ball, i.e. $\mathcal{B}=\{rB_2^n\}_{r>0}$.

In the remainder of Section~\ref{sec:weaksup}, we consider alternatives to this property by restricting the class of bodies under consideration. We first establish that the characterization of the Lebesgue measure like in Theorem~\ref{t:surarecom} does not hold in $\R^2$ when restricting the convex and compact set $L$ to being symmetric; we will not require the point of symmetry to be the origin. A set $L$ is said to be symmetric about a point $x$ if $L-x$ is symmetric.

\begin{reptheorem}{t:contr}
        Let $K$ be a convex body in $\R^2$ and let $\mu$ be the measure with density $|(x_1,x_2)|^2=x_1^2+x_2^2$. Then, for every compact, convex set $L\subset \R^2$ that contains the origin and is symmetric about a point:
        $$\mu^+(\partial(K+L)) \geq \mu^+(\partial K).$$
    \end{reptheorem}

We then study super/sub-modularity for the Gaussian measure.
\begin{reptheorem}{p:mod_log}
The measure $\gamma_1$ is submodular on $\R$ on the class of symmetric convex bodies. In other words, for symmetric intervals $A, B, C\subset\R,$ one has
$$\gamma_1(A+B+C) + \gamma_1(A) \leq \gamma_1(A+B) + \gamma_1(A+C).$$
For $n\geq 2,$ $\gamma_n$ is neither submodular nor supermodular over the class of symmetric, convex sets.
\end{reptheorem}

We then conclude Section~\ref{sec:logsub} by a brief discussion on the property of log-submodularity as a potential future application of our methods. Just like in the volume case, we arrive at the question of local log-submodularity via a differentiation schema as a potential application of our results; in both cases we see that local log-submodularity is strictly stronger than Fenchel's inequality.

This paper is organized as follows. Section~\ref{sec:mixed} is dedicated to proving inequalities for mixed measures that are equivalent to the aforementioned inequalities for mixed volumes, such as Minkowski's first and second inequalities. In Section~\ref{sec:logsup}, we apply mixed measures to study supermodularity and log-submodularity of rotational invariant log-concave measures.

\begin{rem}[Note added in proof]
A seminal result of Pukhlikov and Khovanskii \cite{PK92:1} implies that, for the Borel measure $\nu_n$ on $\mathbb{R}^n$ with density $|x|^2$ and compact, convex sets $K_i\subset \R^n$, the function $$[0,\infty)^r\ni \lambda \longmapsto \nu_n\left(\sum_{i=1}^r \lambda_i K_i\right)$$ is a polynomial of degree at most $n+2$; note that $\nu_2$ corresponds to the measure $\mu$ in our Theorem~\ref{t:contr}. Subsequently, Alesker \cite{Ale99} established several results for the measure $\nu_n$. He showed that if a compact, convex set $K$ contains the origin in its interior, the polynomial $t \mapsto \nu_n(K + t B_2^n)$ has non-negative coefficients and that for origin-symmetric $K_i$, the quartic polynomial $\nu_2(\sum_{i=1}^r \lambda_i K_i)$ also has non-negative coefficients. Finally, Alesker remarked, without proof, the monotonicity, for $0 \in K \subset L$ compact, convex sets, $$\nu_n^+(\partial K) \leq \nu_n^+(\partial L).$$ Clearly, this observation overlaps with our Theorem~\ref{t:contr}; however, neither implies the other. Indeed, in our result, $K$ need not contain the origin, whereas for those $K$ that do contain the origin, his result allows $L$ to be any compact, convex set containing $K$, while ours requires $L$ to be a superset that has $K$ as a Minkowski summand.
\end{rem}

\noindent {\bf Acknowledgments.}
We are grateful to Fedor Nazarov for providing us with a proof of the inequality in Appendix B. We thank Fabian Mu{\ss}nig for the reference used in Remark~\ref{r:fabian}. We thank Dominique Malicet for helpful comments, which lead to Lemma~\ref{l:points} (shortening the proof of Theorem~\ref{t:Radon} drastically) and Theorem~\ref{thm:strict_2}. We thank Alexandros Eskenazis for the insightful question which led to Proposition~\ref{p:zon_ori}. We thank Semyon Alesker for bringing his work (discussed in the preceding note) to our attention. Finally, we thank the anonymous referee, whose thorough remarks greatly enhanced the presentation of this paper.

\section{Inequalities for mixed measures}
\label{sec:mixed}
\noindent Throughout this section, all sets will be subsets of $\R^n$ unless otherwise noted. 
\subsection{Minkowski's inequalities}

To obtain analogues of Minkowski's first and second inequalities for mixed measures, we consider measures with density that have some concavity (like in the volume case), the so-called $F$-concave measures given by \eqref{eq:fcon}. We emphasize that, if $F$ is increasing, like $x^s,$ $s>0,$ then $F\circ \mu$ is a concave function over $\mathcal{C}.$ Likewise, if $F$ is decreasing, like $x^s, s<0,$ then $F\circ \mu$ is a convex function over $\mathcal{C}.$ Additionally, it is not hard to show that, if there is equality in \eqref{eq:fcon} for a single $\lambda\in (0,1),$ then there is equality for every $\lambda \in (0,1).$

A natural question is how $\mu(K)$ and $\mu(K;K)$ are related. For a convex set $K$ containing the origin, notice that, for $t\in [0,1]$, \eqref{eq:arb_mixed_0} yields $\mu(tK;K)=\odv{}{t}\mu(tK),$ where the (one-sided) derivative exists almost everywhere since the function $\mu(tK)$ is monotonic in $t.$ Consequently, integrating yields
\begin{equation}
    \mu(K)=\int_0^1\mu(tK;K)dt.
    \label{eq:measure_mixed_relate_0}
\end{equation}
\noindent We next show Minkowski's first inequality for $F$-concave measures first shown in \cite{Liv19}; the inequality does not require convexity of the sets, only that the given Borel measure has some concavity over a class of Borel sets.

\begin{prop}[Minkowski's first inequality for $F$-concave measures]
\label{prop:first_second}
Let $\mu$ be a Borel measure on $\R^n$ that is $F$-concave on a class of Borel sets $\mathcal{C}.$ Let $K,L\in\mathcal{C}$ with the property that $\mu(K;L)$ and $\mu(K;K)$ exist and are finite, $F$ is differentiable at $x=\mu(K)$, and $F^\prime(\mu(K))\neq 0$. Then:
$$
\mu(K; L) \geq \mu(K; K)+\frac{F(\mu(L))-F(\mu(K))}{F^{\prime}(\mu(K))}.
$$
\end{prop}
\begin{proof}
    As said before, this was first shown in \cite{Liv19}; the proof is adapting a classical proof of Minkowski's first inequality for volume, and so we reproduce it for completeness. For $\epsilon \in (0,1)$, observe that
    \begin{align*}\mu(K+\epsilon L) - \mu(K) &= \mu\left((1-\epsilon)\frac{K}{1-\epsilon}+\epsilon L\right) - \mu(K) 
    \\
    &\geq F^{-1}\left((1-\epsilon)F\left(\mu\left(\frac{K}{1-\epsilon}\right)\right)+\epsilon F(\mu(L))\right) - \mu(K).
    \end{align*}
    Dividing by $\epsilon$ and taking $\liminf$, one obtains (from the chain rule)
    $$\mu(K;L) \geq \frac{1}{F^\prime (\mu(K))}\left(F(\mu(L))-F(\mu(K))\right) + \liminf_{\epsilon \to 0^+}\frac{\mu\left(\frac{K}{1-\epsilon}\right)-\mu(K)}{\epsilon},$$
    where we used that $\odv{F^{-1}(t)}{t}\big|_{t=a}=F^\prime(F^{-1}(a))^{-1}$ with $a=F(\mu(K)).$ Recognizing that the last term is $\mu(K;K)$ completes the proof in this case.

 \end{proof}
\begin{rem}
\label{rem:when_zero}
    In the case that $\mathcal{C}$ is a collection of compact, convex sets, we analyze now the case when $F^\prime(\mu(K))=0$. There are few possible outcomes. Firstly, consider the function $H(t)=F(\mu((1-t)K+tL)).$ Since $F$ is strictly increasing (decreasing), $H$ is concave (convex), all one-sided derivatives exist and therefore $$H^\prime(0)=F^\prime(\mu(K))\lim_{t\to 0^+}\frac{\mu((1-t)K+tL)-\mu(K)}{t}=F^\prime(\mu(K))\left(\mu(K;L)-\mu(K;K)\right)=0.$$
    Here we used the formula \cite[Equation (28)]{FLMZ24:1} which holds for convex sets. Since $H$ is concave (convex), we deduce that $H^\prime(t) \leq 0$ ($H^\prime(t) \geq 0$) for all $t$ (here, the derivatives are one-sided from above). Thus, $H$ is decreasing (increasing). We deduce that $F(\mu(K))=H(0) \geq H(1) = F(\mu(L))$ ($F(\mu(K))=H(0) \leq H(1) = F(\mu(L))$). In either case of monotonicity, we then obtain that $\mu(K)\geq \mu(L)$ for all $L\in\mathcal{C}$.

    We consider the first natural case, when there exists a convex body $M$ containing the origin in $\mathcal{C}$. Then, since $\mu$ is a Borel measure, $t\mapsto \mu(K+tM)$ is non-decreasing in $t$. But, we have already shown that $\mu(K) \geq \mu(L)$ for all $L$, including $L=K+tM$ (remember, $\mathcal{C}$ is closed under sums and dilations). Thus, we have either that $K$ contains the support of $\mu$ or reached a contradiction, and so $F^\prime(\mu(K))=0$ can never happen under usual circumstances. In this instance, the inequality should be interpreted as $\mu(K;M) = 0$.

    We consider our next case, when such an $M$ may not exist and there does not exist an $L$ such that $\mu(K)=\mu(L)$. Then, from  $\mu(K) > \mu(K+tL)$ for every $L \in {\mathcal C}$, we have $\mu(K;L)\le 0$. In particular, we deduce that $\mu(K;K)\le 0$. We claim in this instance that
    $$\frac{F(\mu(L))-F(\mu(K))}{F^{\prime}(\mu(K))}
    $$ 
    should be interpreted as $-\infty$. Indeed, if $F$ is increasing, then for every $L\in\mathcal{C}$, $F(\mu(L))-F(\mu(K)) < 0$ and, since $F^\prime$ is non-negative when it exists, $1/F^\prime(\mu(K))$ should be interpreted as $+\infty$. Similarly, if $F$ is decreasing, then $F(\mu(L))-F(\mu(K)) > 0$ and, since $F^\prime$ is non-positive when it exists, $1/F^\prime(\mu(K))$ should be interpreted as $-\infty$.

    Our final case is when such an $M$ may not exist and there does exist an $L$ such that $\mu(K)=\mu(L)$. Then, $H(0)=H(1)$ and $H(0)\geq H(t)$ for $t\in (0,1)$. Thus, $H(t)$ is constant. Hence, $\mu((1-t)K + tL)$ is constant, which means, taking the derivative again using \cite[Equation 28]{FLMZ24:1}, that $\mu(K;K)=\mu(K;L)$. This is how the inequality should be understood in this very specific situation.
\end{rem}

\noindent An important consequence of Proposition~\ref{prop:first_second} is that, if additionally one has $\mu(K)=\mu(L),$ then $\mu(K;L)\geq \mu(K;K).$ One can verify that the Minkowski's first inequality becomes \eqref{eq:second_in} when $\mu=\lambda_n.$ Indeed, if $\mu$ is $\alpha$-homogeneous, $\alpha>0$ (that is, $\mu(tK)=t^{\alpha}\mu(K)$),
then one can readily verify that $\mu(tK,L)=t^{\alpha-1}\mu(K;L).$ Consequently, one deduces for such $\mu$ via \eqref{eq:measure_mixed_relate_0} that $\mu(K)=\frac{1}{\alpha}\mu(K;K).$ Thus, Minkowski's first inequality from Proposition~\ref{prop:first_second} for an $\alpha$-homogeneous, $1/\alpha$-concave measure, reads as
$$\left[\frac{1}{\alpha}\mu(K;L)\right]^{\alpha}\geq \mu(K)^{\alpha-1}\mu(L),$$
which was first shown by Milman and Rotem \cite{MR14}. 

Working under the minor assumption that the class $\mathcal{C}$ in Proposition~\ref{prop:first_second} is a class of convex sets, we can obtain Minkowski's second inequality. This Minkowski's second inequality will actually be distinct from the version obtained in the prequel, \cite{FLMZ24:1}. We also note, by writing out the limit definition of the derivative, that for a Borel measure $\mu$ on $\R^n$ and Borel sets $A,B$ and $C$ in $\R^n$ such that $\mu(A;B,C)$ from Definition~\ref{def:mixed} exists, one has
\begin{equation}
    \mu(A;B,C)=\lim_{s\to 0^+}\frac{\mu(A+sB;C)-\mu(A;C)}{s}.
    \label{eq:limit_def_alt}
\end{equation}
In particular, one sees that if $A$ is convex
$$\mu(tA;A,C)=\odv{\mu(tA;C)}{t}.$$
We therefore deduce that
\begin{equation}
    \mu(A;C)=\int_0^1\mu(tA;A,C)dt.
    \label{eq:int_formula_double_mixed}
\end{equation}

\noindent Notice that if $A$ and $B$ are convex, then
\begin{align}
    \mu(A;B,B) &= \pdv{}{s,t}\mu(A+sB+tB)(0,0) \nonumber \\&=\pdv{}{s,t}\mu(A+(s+t)B)(0,0)=\odv[2]{}{s}\mu(A+sB)\big|_{s=0}. \label{eq:2BorNot}
\end{align}
We now state Minkowski's second inequality for $F$-concave measures. 

\begin{prop}[Minkowski's second inequality for $F$-concave measures]
\label{prop:second}
    Let $\mu$ be a Borel measure on $\R^n$ that is $F$-concave on a class of convex sets $\mathcal{C}$, with the additional property that, for every $K,L\in\mathcal{C}$, $\mu(K;L),$ $\mu(K;K)$ and $\mu(K;L,L)$ exist and are finite, $F$ is twice differentiable at $x=\mu(K)$, and $F^\prime(\mu(K))\neq0$. Then:
    \begin{equation}-\frac{F^{\prime\prime}(\mu(K))}{F^\prime(\mu(K))}\mu^2(K;L) \geq \mu(K;L,L).
    \label{eq:min_second_F}
    \end{equation}
\end{prop}
\begin{proof}
    Consider the function $H(\lambda)=F(\mu(K+\lambda L))$. If $F$ is increasing, then, by hypothesis, this function is concave in $\lambda$. Indeed, write
    \begin{align*}
        F\circ \mu(K+[(1-\lambda)x+\lambda y] L) &= F\circ \mu((1-\lambda)[K+xL] + \lambda [K+yL])
        \\
        &\geq (1-\lambda) F\circ \mu(K+xL) + \lambda F\circ \mu(K+yL).
    \end{align*}
     From the chain rule, one has that
    $$H^\prime(\lambda)=F^{\prime}(\mu(K+\lambda L))\odv{}{t}\mu(K+t L)\big |_{t=\lambda}.$$
     The result then follows from taking another derivative and using \eqref{eq:2BorNot} in conjunction with the fact that $H^{\prime\prime}(0) \leq 0$. The result for when $F$ is decreasing is similar.
\end{proof}
\noindent 
\begin{rem}
    If $F^{\prime}(\mu(K))=0$, then Proposition~\ref{prop:second} should be understood as $F^{\prime\prime}(\mu(K)) \mu^2(K;L) \leq 0$ if $F$ is increasing and $F^{\prime\prime}(\mu(K)) \mu^2(K;L) \geq 0$ if $F$ is decreasing. 
\end{rem}

We isolate the cases when $\mu$ is $s$-concave; recall that $s=0$ is log-concavity.
\begin{prop}[Minkowski's second inequality for $s$-concave measures]
\label{prop:second_s}
    Let $\mu$ be a Borel measure on $\R^n$ that is $s$-concave, $s\in (-\infty,1),$ on a class of convex sets $\mathcal{C}$, with the additional property that $\mu(K;L),$ $\mu(K;K)$ and $\mu(K;L,L)$ exist and are finite for every $K,L\in\mathcal{C}$. Then
    \begin{equation}\mu(K)\mu(K;L,L) \leq (1-s)\mu^2(K;L).
    \label{eq:min_second}
    \end{equation} In particular, if $\mu$ is log-concave, one has $\mu(K)\mu(K;L,L) \leq \mu^2(K;L).$
\end{prop}

\noindent An interesting example of Proposition~\ref{prop:second_s} is  $s=1/n$, $\mathcal{C}$ is all symmetric convex bodies, and $\mu\in\mathcal{M}_n$, the class of measures defined in \eqref{eq:CER} (e.g. $\gamma_n$).

Observe that, for $\mu(A;B,C)$ when $\mu$ is $\alpha$-homogeneous, one has $\mu(tA;B,C)=t^{\alpha-2}\mu(A;B,C).$ 
Indeed,
\begin{align*}
\mu(tA;B,C)&=\lim_{\epsilon\to 0^+}\frac{\mu(tA+\epsilon B;C)-\mu(tA;C)}{\epsilon}=t^{\alpha-1}\lim_{\epsilon\to 0^+}\frac{\mu(A+\epsilon t^{-1}B;C)-\mu(A;C)}{\epsilon}
\\
&=t^{\alpha-2}\lim_{\epsilon\to 0^+}\frac{\mu(A+\epsilon t^{-1}B;C)-\mu(A;C)}{\epsilon t^{-1}}=t^{\alpha-2}\mu(A;B,C).
\end{align*}
Inserting into \eqref{eq:int_formula_double_mixed}, we obtain for an $\alpha$-homogeneous measure, $\alpha\in(0,1)\cup(1,\infty),$ that
\begin{equation}
    \mu(A;C)=\frac{1}{\alpha-1}\mu(A;A,C).
    \label{eq:dobule_homogeneous}
\end{equation}
Combining this with $\mu(A)=\frac{1}{\alpha}\mu(A;A),$ we obtain that in this instance one has
$$\mu(A;A,A)=\alpha(\alpha-1)\mu(A).$$

\subsection{Fenchel-type inequalities}
We start this subsection by showing the measure case of \eqref{eq:reverse}.
\begin{thm}[Reverse Minkowski's quadratic inequality for mixed measures]
\label{t:min_quad}
Let $\mu$ be a Borel measure on $\R^n$ that is $F$-concave on a class of convex sets $\mathcal{C}$. Then, for $A,B,C\in\mathcal{C}$ such that $F(x)$ is twice differentiable at $x=\mu(A)$,:
\begin{align*}
F^{\prime}(\mu(A))^2\mu(A; B,B)&\mu(A;C,C)+F^{\prime}(\mu(A))F^{\prime\prime}\left(\mu(A)\right)\left(\mu^2(A;B)\mu(A;C,C)+\mu^2(A;C)\mu(A; B,B)\right) 
\\
&\geq F^{\prime}(\mu(A))^2\mu^2(A;B,C)+2F^{\prime}(\mu(A))F^{\prime\prime}\left(\mu(A)\right)\mu(A;C)\mu(A;B)\mu(A;B,C),
\end{align*}
whenever all terms exist and are finite.
\end{thm}
\begin{proof}
Consider the function of two variables
\begin{equation}
    \label{eq:fen_mu}
    f(s,t)=F\left(\mu(A+sB+tC)\right).
\end{equation}
Suppose, for a moment, that $F$ is an increasing function. Observe that, since $\mu$ is $F$-concave, one has, with $x=(s_1,t_1),y=(s_2,t_2)$, for every $\lambda\in[0,1]:$
\begin{align*}
f((1-\lambda)x+\lambda y)&=F\left(\mu\left((1-\lambda)\left(A+s_1B+t_1C\right)+\lambda\left(A+s_2B+t_2C\right)\right)\right)
\\
&\geq (1-\lambda)F\left(\mu(A+s_1B+t_1C)\right)+\lambda F\left(\mu(A+s_2B+t_2C)\right)
\\
&=(1-\lambda)f(s_1,t_1)+\lambda f(s_2,t_2)
\end{align*}
(here, we used convexity of $A$ to write $A=(1-\lambda)A + \lambda A$ and re-arrange). Hence, $f(s,t)$ is concave over $\R^2$. Similarly, if $F$ is a decreasing function, then $f(s,t)$ is convex over $\R^2$. In either case, the determinant of the Hessian of $f(s,t)$ is non-negative at the origin. Thus,
\begin{equation}\left(\pdv{f}{s,t}(0,0)\right)^2\le \pdv[2]{f}{s}(0,0)\pdv[2]{f}{t}(0,0).
\label{eq:mixed_deriv_fen}
\end{equation}
Notice that $\pdv{f}{s}(s,t)=F^\prime\left(\mu(A+sB+tC)\right)\pdv{}{s}\mu(A+sB+tC)$ and so from \eqref{eq:2BorNot} one has
\begin{align*}\pdv[2]{f}{s}(0,0)=F^{\prime\prime}\left(\mu(A)\right)&\left(\pdv{}{s}\mu(A+sB+tC)\bigg|_{(s,t)=(0,0)}\right)^2
\\
&+F^\prime\left(\mu(A)\right)\left(\pdv[2]{}{s}\mu(A+sB+tC)\bigg|_{(s,t)=(0,0)}\right)
\\
=F^{\prime\prime}\left(\mu(A)\right)&\mu(A;B)^2+F^\prime\left(\mu(A)\right)\mu(A;B,B),
\end{align*}
and similarly for $\pdv[2]{f}{t}(0,0).$ But also, from the definition of $\mu(A;B,C),$ one has
\begin{align*}
    \pdv{f}{t,s}(0,0)=F^{\prime\prime}\left(\mu(A)\right)&\left(\pdv{}{s}\mu(A+sB+tC)\bigg|_{(s,t)=(0,0)}\right)\left(\pdv{}{t}\mu(A+sB+tC)\bigg|_{(s,t)=(0,0)}\right)
    \\
    &+F^\prime\left(\mu(A)\right)\left(\pdv[1]{}{t,s}\mu(A+sB+tC)\bigg|_{(s,t)=(0,0)}\right)
    \\
    =F^{\prime\prime}\left(\mu(A)\right)&\mu(A;B)\mu(A;C)+F^\prime\left(\mu(A)\right)\mu(A;B,C).
\end{align*}
Inserting these computations into \eqref{eq:mixed_deriv_fen} yields the result.

\end{proof}

\noindent We present below the case where $\mu$ is $s$-concave on a class of convex sets as a corollary.
\begin{cor}
\label{cor:AFs}
Let $\mu$ be a Borel measure on $\R^n$ that is $s$-concave, $s\in (-\infty,0)\cup(0,1)$ on a class $\mathcal{C}$ of convex sets. Then, for $A,B,C\in\mathcal{C}$ with $\mu(A)\neq 0$:
\begin{align*}
s\mu(A)&\mu(A; B,B)\mu(A;C,C)
\\
&-s\left(1-s\right)\left(\mu^2(A;B)\mu(A;C,C)+\mu^2(A;C)\mu(A; B,B)\right) 
\\
&\geq s\mu(A)\mu^2(A;B,C)-2s\left(1-s\right)\mu(A;C)\mu(A;B)\mu(A;B,C),
\end{align*}
whenever all terms exist and are finite. If $\mu$ is log-concave $(s=0)$, then under the same assumptions,
\begin{align*}
\mu(A)&\mu(A; B,B)\mu(A;C,C)-\left(\mu^2(A;B)\mu(A;C,C)+\mu^2(A;C)\mu(A; B,B)\right) 
\\
&\geq \mu(A)\mu^2(A;B,C)-2\mu(A;C)\mu(A;B)\mu(A;B,C).
\end{align*}
\end{cor}

One can verify that the inequality in Corollary~\ref{cor:AFs} reduces to \eqref{eq:reverse} when $\mu$ is the Lebesgue measure, $s=1/n$, and $\mathcal{C}$ is the class of convex bodies in $\R^n$. One can also take $s=1/n$ and $\mathcal{C}$ to be all symmetric convex bodies if $\mu\in\mathcal{M}_n$. In both of those cases, every $A$ such that $\mu(A)=0$ can be approximated by a sequence in $\mathcal{C}$ that have positive measure. Thus, in Corollary~\ref{cor:AFs}, the assumption $\mu(A)\neq 0$ can be removed. We can continue developing the result.
\begin{lem}
\label{lem:fen}
   Let $\mu$ be a Borel measure on $\R^n$ that is $s$-concave, $s\in [0,1),$ on a class $\mathcal{C}$ of convex sets, each of which contain the origin. Then, for $A,B,C\in\mathcal{C}$ with $\mu(A)\neq 0$, we have

\begin{equation} (1-\sqrt{\mathcal{D}})(1-s) \leq \frac{\mu(A)\mu(A;B,C)}{\mu(A;B)\mu(A;C)} \leq (1+ \sqrt{\mathcal{D}})(1-s) ,
\label{eq:Fenchel_poly}
\end{equation}
where
$$\mathcal{D}:=\left(1-\frac{\mu(A)}{1-s}\frac{\mu(A;B,B)}{\mu(A;B)^2}\right)\left(1-\frac{\mu(A)}{1-s}\frac{\mu(A;C,C)}{\mu(A;C)^2}\right)\geq 0,$$
whenever all terms exist; recall that $s=0$ means $\mu$ is log-concave.
\end{lem}
\begin{proof}
From the hypothesis that each set in $\mathcal{C}$ contains the origin, we have that $\mu(A;B) \geq 0$ for every $A,B\in\mathcal{C}$, since $A+\epsilon B \supset A$ yields $\mu(A+\epsilon B) - \mu(A) \geq 0,$ and thus $\mu(A;B)$ is a limit of non-negative quantities. Next, define the following variables
$$x=\frac{\mu(A;B,B)}{\mu(A;B)^2}\frac{\mu(A)}{1-s}, \quad y = \frac{\mu(A;C,C)}{\mu(A;C)^2}\frac{\mu(A)}{1-s}$$
and 
$$\beta = x+y -xy.$$
Consider the polynomial
\begin{equation}
\label{eq:poly}
    P(t)=t^2-2t+\beta.
\end{equation}
Clearly, $P(t) >0$ for $t$ large enough. However, Corollary~\ref{cor:AFs} precisely states upon re-arrangement (using that $\mu(A;B) \geq 0$ and $s>0$):
$$P\left(\frac{\mu(A;B,C)\mu(A)}{(1-s)\mu(A;B)\mu(A;C)}\right) \leq 0.$$
Therefore, we deduce that the polynomial defined by \eqref{eq:poly} must have real roots, and, since the choice of $t=\frac{\mu(A;B,C)\mu(A)}{(1-s)\mu(A;B)\mu(A;C)}$ yields a non-positive value for $P(t),$ we must have that this choice of $t$ is between the two roots, i.e.
$$1-\sqrt{1-\beta} \leq \frac{\mu(A)\mu(A;B,C)}{(1-s)\mu(A;B)\mu(A;C)} \leq 1+\sqrt{1-\beta}.$$
Here, we know that $1-\beta \geq 0$, since the discriminant of $P(t)$ is non-negative, i.e. $4-4\beta\geq 0$. Setting $\mathcal{D}:=1-\beta$ yields \eqref{eq:Fenchel_poly}. However, notice that
$$\mathcal{D}=1-\beta = 1-x-y +xy=(1-x)(1-y);$$
inserting the definitions of $x$ and $y$ yields the formula for $\mathcal{D}$.
\end{proof}
Using Lemma~\ref{lem:fen}, we obtain a Fenchel-type inequality for $s$-concave measures.
\begin{thm}[Fenchel-type inequality for $s$-concave measures]
    Let $\mu$ be a Borel measure on $\R^n$ that is $s$-concave, $s\in [0,1)$ on a class $\mathcal{C}$ of convex sets, each of which contain the origin. Then, for $A,B,C\in\mathcal{C}$ such that $\mu(A;B,B),\mu(A;C,C)$ and $\mu(A;B,C)$ exist, one has
    $$\frac{\mu(A)\mu(A;B,C)}{\mu(A;B)\mu(A;C)} \leq (1-s) \left(2-\frac{\mu(A)}{2(1-s)}\left(\frac{\mu(A;B,B)}{\mu(A;B)^2}+\frac{\mu(A;C,C)}{\mu(A;C)^2}\right)\right)$$
    if $\mu(A;B)$ and $\mu(A;C)$ are finite and $\mu(A)\neq 0$; recall $s=0$ means $\mu$ is log-concave.
    \label{t:fen}
\end{thm}
\begin{proof}
    Under the notation of Lemma~\ref{lem:fen}, the arithmetic-geometric mean inequality yields
    $$\sqrt{\mathcal{D}} \leq 1 - \frac{\mu(A)}{2(1-s)}\left(\frac{\mu(A;B,B)}{\mu(A;B)^2}+\frac{\mu(A;C,C)}{\mu(A;C)^2}\right),$$
    which we can use, since Proposition~\ref{prop:second_s} shows that $\mathcal{D}$ is the product of two non-negative numbers. From the right-hand side of \eqref{eq:Fenchel_poly}, we obtain
    $$\frac{\mu(A)\mu(A;B,C)}{\mu(A;B)\mu(A;C)} \leq (1-s) \left(2-\frac{\mu(A)}{2(1-s)}\left(\frac{\mu(A;B,B)}{\mu(A;B)^2}+\frac{\mu(A;C,C)}{\mu(A;C)^2}\right)\right)$$
    as claimed.
    \end{proof}
   If, moreover, $\mu(A;B,B)$ and $\mu(A;C,C)$ are non-negative (e.g. in the case of volume) then, one has that
    \begin{equation}-\frac{\mu(A)}{2(1-s)}\left(\frac{\mu(A;B,B)}{\mu(A;B)^2}+\frac{\mu(A;C,C)}{\mu(A;C)^2}\right) \leq 0.
    \label{eq:fen_crit}
    \end{equation}
    \noindent Theorem~\ref{t:fen} then implies in this instance
    \begin{equation}
    \mu(A)\mu(A;B,C) \leq 2 (1-s) \mu(A;B)\mu(A;C).
    \label{eq:sup_fen}
    \end{equation}
Let $A$ be an arbitrary convex body, $B$ and $C$ be arbitrary compact, convex sets, $\mu=\lambda_n$ the Lebesgue measure, and $s=1/n$. Then \eqref{eq:fen_crit} is always true. Thus, \eqref{eq:sup_fen} becomes $$\vol_n(A)\lambda_n(A;B,C)\leq 2 \frac{n-1}{n}\lambda_n(A;B)\lambda_n(A;C).$$
Replacing mixed measures with mixed volumes (recall $\lambda_n(A;B)=nV(A[n-1],B)$ and $\lambda_n(A;B,C)=n(n-1)V(A[n-2],B,C)$) then yields the classical Fenchel's inequality \eqref{eq:fen}. That is, Theorem~\ref{t:fen} yields a new proof of \eqref{eq:fen}.

\section{Applications of mixed measures}
\label{sec:logsup}
\subsection{Supermodularity over convex sets characterizes the Lebesgue Measure}
\label{sec:lebclass}
We recall that a Borel measure $\mu$ on $\R^n$ is supermodular over a class of compact sets $\mathcal{C}$ if for $A,B,C\in\mathcal{C}$ one has
    $$
        \mu(A+B+C) +\mu(A) \geq \mu(A+B) + \mu(A+C).$$
Similarly, a measure is  submodular if 
$$\mu(A+B+C) +\mu(A) \leq \mu(A+B) + \mu(A+C).$$
We next establish that any supermodular measure is the weak limit of a sequence of supermodular measures with continuously differentiable density using a measure-theory-based approach. In Appendix C, Section~\ref{sec:density}, we give a more geometric-based approach.
\begin{prop}
\label{p:sup_mod_shift}
    Let $n\geq 1$. Consider a Borel measure $\mu$ and Borel sets $A,B,C\subset \R^n$. Suppose that, for every $z\in \R^n$,
      \begin{equation}
        \mu(A-z+B+C) +\mu(A-z) \geq \mu(A-z+B) + \mu(A-z+C).
        \label{eq:sup_mod_shift}
    \end{equation}
    Then there exists a sequence of Borel measures $\{\mu_j\}$, each with continuously differentiable density, such that $\mu_j$ converges to $\mu$ weakly and \eqref{eq:sup_mod_shift} holds with $\mu$ replaced by $\mu_j$ for each $j$.
\end{prop}
\begin{proof}
Consider \eqref{eq:sup_mod_shift} not for $z$ but for $z+z_0$ ($z,z_0\in\R^n$). Next, multiply the inequality through by an arbitrary, non-negative $g\in L^1(\R^n,\mu)\cap L^1(\R^n,dx)$ and integrate over $\R^n$ (in the variable $z$) to obtain
\begin{align*}\int_{\R^n}&\int_{A+B+C-(z+z_0)}d\mu(x)g(z)dz + \int_{\R^n}\int_{A-(z+z_0)}d\mu(x)g(z)dz
\\
&\geq \int_{\R^n}\int_{A+B-(z+z_0)}d\mu(x)g(z)dz + \int_{\R^n}\int_{A+C-(z+z_0)}d\mu(x)g(z)dz.
\end{align*}
We then use Fubini's theorem and the variable substitution $z\mapsto z-x$ to obtain
\begin{align*}\int_{A+B+C-z_0}\int_{\R^n}&g(z-x)dzd\mu(x) + \int_{A-z_0}\int_{\R^n}g(z-x)dzd\mu(x)
\\
&\geq \int_{A+B-z_0}\int_{\R^n}g(z-x)dzd\mu(x) + \int_{A+C-z_0}\int_{\R^n}g(z-x)dzd\mu(x).
\end{align*}
Defining the function $\phi_g(z)=(\mu \ast g)(z)=\int_{\R^n}g(z-x)d\mu(x)$
 yields
\begin{align*}\int_{A+B+C-z_0}\phi_g(z)dz + \int_{A-z_0}\phi_g(z)dz
\geq \int_{A+B-z_0}\phi_g(z)dz + \int_{A+C-z_0}\phi_g(z)dz.
\end{align*}
Let $\mu_g$ be the Borel measure with density $\phi_g$. Then, we have shown that \eqref{eq:sup_mod_shift} holds with $\mu$ replaced by $\mu_g$ and $z$ replaced by $z_0$.
In particular, we can consider a sequence of measures $\{\mu_{g_j}\}$, where $g$ is a mollifier and $g_j(z)=j^ng(zj)$. One has that $\phi_{g_j}$ is continuously differentiable for all $j=1,\dots,$ $\lim_{j\to\infty} \mu_j=\mu$ weakly, and each $\mu_{g_j}$ satisfies \eqref{eq:sup_mod_shift}.
\end{proof}

We are now in position to show the main theorem of this section: if a Radon measure $\mu$ is supermodular over all convex bodies, then it must be a multiple of the Lebesgue measure. 
\begin{thm}\label{t:Radon}
    Let $n\geq 1$. Let $\mu$ be a Radon measure that is supermodular (i.e. satisfies \eqref{eq:sup_mod}) over the class of all convex bodies in $\R^n$. Then $\mu$ is a multiple of the Lebesgue measure.
\end{thm}

Theorem~\ref{t:Radon} is actually a direct consequence of Lemma~\ref{l:points} below, where we show supermodularity of a Radon measure $\mu$ over a class of homothets of a fixed convex body $K$ is enough to conclude $\mu$ is a multiple of the Lebesgue measure. First, we need some facts about Lebesgue's differentiation theorem. Let $B_2^n(r,x)=x+rB_2^n$.The usual Lebesgue differentiation theorem states that, for an integrable function $f$, $f(x)=\lim_{r\to 0^+}\frac{1}{\vol_n(B_2^n(r,x))}\int_{B_2^n(r,x)}f(y)dy$ almost everywhere. 
There is a more general extension of Lebesgue's differentiation theorem (e.g. in  Guzm\'an \cite{Guz75:book}, Folland \cite{Fol99:book} and \cite[p. 108]{SS05:book}): for almost every $x\in\R^n,$
$$f(x)=\lim_{r\to 0^+}\frac{1}{\vol_n(E_r(x))}\int_{E_r(x)}f(y)dy$$
for every family of Borel sets $\{E_r(x)\}_{r>0}$ that shrink nicely to $x$. Recall that such a family shrinks nicely to $x$ if $E_r(x) \subseteq B_2^n(r,x)$, and yet there exists a constant $\alpha(x)$ independent of $r$ such that $\vol_n(E_r(x))\geq \alpha(x) \vol_n(B_2^n(r,x))$.  For a convex body $K$, set $K(r,z)=rK+z$ for $r>0$ and $z\in\R^n$. Then, since $K$ has non-empty interior and is compact and star-shaped, $\{K(r,x)\}_{r>0}$ shrinks nicely to $x$.

\begin{lem}
\label{l:points}
    Let $\mu$ be a Radon measure on $\R^n$, $n\geq 1$, with the following property: there exists a convex body $K$ such that, for every $x,y,z\in\R^n$ and $r$ small enough (say, $r\in (0,1)$),
    \begin{equation}\mu(K(r,z))+\mu(K(r,z)+x+y) \geq \mu(K(r,z)+x)+\mu(K(r,z)+y).
    \label{eq:points}
    \end{equation}    
    \noindent Then $\mu$ is a constant multiple of the Lebesgue measure.
\end{lem}
\begin{proof}
    Assume that $\mu$ has continuously differentiable density; denote this density by $\varphi$. We use Lebesgue's differentiation theorem along the basis $\{K(r,x)\}_{r>0}$. We divide \eqref{eq:points} by $r^{n}\vol_n(K)$ and send $r$ to $0$ to obtain from the Lebesgue differentiation theorem that
\begin{equation}
        \varphi(z)+\varphi(z+x+y) \geq \varphi(z+x)+\varphi(z+y).
        \label{eq:points_2}
        \end{equation}
    Set $y=te_i$ for $t>0$ and $e_i$ the $i$th canonical basis vector. Then, \eqref{eq:points_2} yields
    \begin{equation*}
        \varphi(z+x+te_i) - \varphi(z+x) \geq \varphi(z+te_i)-\varphi(z).
        \end{equation*}
        Dividing by $t$ and sending $t$ to $0$ then yields $\nabla_i \varphi(z+x) \geq \nabla_i\varphi(z)$, where $\nabla_i \varphi(x)=\langle \nabla \varphi(x),e_i \rangle$ is the directional derivative. Since $x$ is arbitrary, we deduce $\nabla_i \varphi(x)$ is a constant function. Since this is true for all $i$, we must have that $\varphi$ is an affine function. Since $\varphi$ is the density of a Borel measure, it is non-negative; a non-negative, affine function must be constant, and the claim follows.

For the general case, we first use Proposition~\ref{p:sup_mod_shift} to  approximate $\mu$ with a sequence of Borel measures $\mu_j$ that have continuously differentiable densities and satisfy \eqref{eq:points}. Then, $\mu_j=c_j\vol_n$ for a sequence of non-negative $c_j$. Using that $\mu$ is a Radon measure, we get that the sequence $\{c_j\}$ is bounded. Thus, up to passing to a subsequence, $c_j\to c$ for some $c>0$. On the other-hand, $\mu_j\to\mu$ weakly (which we recall, means $\int_{\R^n}f(x)d\mu_j(x)\to \int_{\R^n}f(x)d\mu(x)$ for every compactly supported, continuous function $f$). Consequently, we obtain that $\mu$ coincides with $c\vol_n$ on open sets, and thus $\mu=c\vol_n$. 
\end{proof}

\begin{rem}
    As pointed out by the referee, the above procedure also works for submodularity, i.e., if a measure $\mu$ satisfies \eqref{eq:points}, but with the opposite inequality, then it would have to be the Lebesgue measure. But the Lebesgue measure is not submodular. Consequently, there are no measures that are submodular over the set of all convex, compact sets.
\end{rem}

We conclude this subsection by using the main theorem to obtain results for $\mu(A;B,C)$. We first establish the localized form of supermodularity, which is a more general form of Proposition~\ref{p:sup_deriv_equiv_0} from the introduction. Indeed, Proposition~\ref{p:sup_deriv_equiv_0} is merely the case when $\mathcal{A}=\mathcal{B}=\mathcal{C}$ in Proposition~\ref{p:sup_deriv_equiv}.

\begin{prop}\label{p:sup_deriv_equiv}
    Let $n\geq 1$ and consider a Radon measure $\mu$ on $\R^n$. Let $\mathcal{A}, \mathcal{B}$ and $\mathcal{C}$ be classes of Borel sets such that $\mathcal{B},\mathcal{C} \subset \mathcal{A}$ and the limits in the definitions of $\mu(A;B)$ and $\mu(A;B,C)$ exist for every $A\in\mathcal{A}, B\in\mathcal{B}, C\in\mathcal{C}$. Then the following three families of inequalities are equivalent: 
    \begin{enumerate}
        \item $\mu(A+B+C) +\mu(A) \geq \mu(A+B) + \mu(A+C),$ for every $A\in\mathcal{A}, B\in\mathcal{B}$ and $C\in\mathcal{C}$,
        \item  $\mu(A+C;B)\geq \mu(A;B)$, for every $A\in\mathcal{A}, B\in\mathcal{B}$ and $C\in\mathcal{C}$,
        \item $\mu(A;B,C) \geq 0$, for every $A\in\mathcal{A}, B\in\mathcal{B}$ and $C\in\mathcal{C}$.
    \end{enumerate}
\end{prop}

\begin{proof}
    Define the function $g_{A,B,C}(s,t)=g(s,t)=\mu(A+sB+tC).$ Since $\mu(A;B,C)$ and $\mu(A;B)$ exist, this implies that $g$ is twice differentiable and
    $$\frac{\partial}{\partial s }g(0,t)=\mu(A+tC;B) \quad \text{and} \quad \frac{\partial^2}{\partial s \partial t}g(0,0)=\mu(A;B,C).$$
    Suppose the first inequality holds. Then, by replacing $B$ and $C$ with $sB$ and $tC$ respectively in \eqref{eq:sup_mod}, we know that $g$ satisfies
    $$g(s,t) + g(0,0) \geq g(s,0) + g(0,t).$$
    This implies that
    $$\mu(A+tC;B)=\frac{\partial}{\partial s }g(0,t)\geq \frac{\partial}{\partial s }g(0,0)=\mu(A;B).$$
    Setting $t=1$ in the above inequality yields the second inequality. Moreover, this inequality then implies the third inequality $\mu(A;B,C)=\frac{\partial^2}{\partial s \partial t}g(0,0) \geq 0$. We now complete the loop. Suppose the third inequality holds, i.e. $\mu(A;B,C) \geq 0$ for every $A\in\mathcal{A}, B\in\mathcal{B}$ and $C\in\mathcal{C}$. Since $\mathcal{A}$ is closed under Minkowski summation and dilations, and $\mathcal{A}$ contains $\mathcal{B}$ and $\mathcal{C}$, we know that $\frac{\partial^2}{\partial s \partial t}g(s_0,t_0)=\mu(A+s_0B+t_0C;B,C) \geq 0$ for every $s_0,t_0 \geq 0.$ Thus, we get
    \[
g(s_0,t_0)-g(s_0,0)-g(0,t_0)+g(0,0)=\int_0^{s_0}\int_0^{t_0}\frac{\partial^2}{\partial s\partial t}g(s,t)dsdt \geq 0,
\]
    which is the first inequality by setting $t_0=s_0=1$.
 \end{proof}
Then, the following corollary is immediate from Theorem~\ref{t:Radon} and Proposition~\ref{p:sup_deriv_equiv}.
 \begin{cor}
     Let $\mu$ be a Radon measure on $\R^n$, $n\geq 1$ that is not a multiple of the Lebesgue measure, such that $\mu(A;B,C)$ exists for every triple of convex bodies $A,B$ and $C$ (e.g. $\mu$ has density). Then, for every $A$, there exist $B$ and $C$ such that
     $$\mu(A;B,C) <0.$$
 \end{cor}

\subsection{Weaker forms of supermodularity}
\label{sec:weaksup}
This section is dedicated to proving the following theorem.
\begin{thm}\label{t:surarecom}
    For $n\geq 1$, let $\mu$ be a Radon measure on $\R^n$ such that, for every convex body $K$ and compact, convex set $L$, 
    $$\mu^+(\partial (K+L)) \geq \mu^+(\partial K).$$
    Then $\mu$ is a multiple of the Lebesgue measure.
\end{thm}

\noindent We treat the $n=1$ case separately, where we will actually further restrict the admissible $L$.
\begin{lem}
\label{l:dim_1}
    Let $\mu$ be a Borel measure with continuous density on $\R$ such that, for every closed, compact interval $K$ and closed, compact interval $L$ containing $0$, one has
    \begin{equation}\mu^+(\partial (K+L)) \geq \mu^+(\partial K).
    \label{eq:with0}
    \end{equation}
    Then $\mu$ is a multiple of the Lebesgue measure.
\end{lem}
\begin{proof}
    Denote by $\varphi$ the density of $\mu$. Let $K=[a,b]$ and $L=[c,d]$, with $c\leq 0 \leq d$. From direct computation, we obtain $\mu^+(\partial K)=\phi(a)+\phi(b)$. Similarly, $\mu^+(\partial (K+L))=\phi(a+c)+\phi(b+d)$. Inserting into \eqref{eq:with0} and rearranging yields
    $$\phi(b+d)-\phi(b)\geq \phi(a)-\phi(a+c).$$
    Let $c=0$. Then, one obtains
    $\phi(b+d)\geq \phi(b)$. Since $d>c=0$, this means $\phi$ is an increasing function. On the other-hand, consider the case when $d=0$. Then, one obtains
$\phi(a+c)\geq \phi(a)$. But, since $c<d=0$, this means $a+c <a$ and thus $\phi$ is a decreasing function. Since $\phi$ is both increasing and decreasing, it must be a constant.
\end{proof}
To prove Theorem~\ref{t:surarecom} for higher dimensions, we note that, if $\mathcal{B}=\{rB_2^n\}_{r\geq 0}$ in Proposition~\ref{p:sup_deriv_equiv}, then item 2 in that proposition reads as $\mu^+(\partial (A+C))\geq \mu^+(\partial A).$
    
    \begin{proof}[Proof of Theorem~\ref{t:surarecom}]Our hypothesis is equivalent to the item 1 of Proposition~\ref{p:sup_deriv_equiv} holding for every $A,C$ convex bodies and $B\in \mathcal{B}=\{rB_2^n\}_{r\geq 0}$. By using Proposition~\ref{p:sup_mod_shift} like in the proof of Lemma~\ref{l:points}, we can suppose that $\mu$ has a continuous density $\varphi$ of $\mu$. Then, the $n=1$ case follows from Lemma~\ref{l:dim_1}.
    
For dimensions greater than $1$, fix arbitrary $x,y\in\R^n$ and let $L=\{x-y\}$ and replace $K$ with $K+y$. Then, by hypothesis we have that
    $$\mu^+(\partial K +x) \geq \mu^+(\partial K+y).$$
    By switching $x$ and $y$, we deduce that the function $x\mapsto \mu^+(\partial K+x)$ is constant. Now we consider the case $K=rB_2^n$ and send $r$ to $0$:
    \begin{equation}
    \frac{1}{r^{n-1}\vol_{n-1}(\s^{n-1})}\mu^+(r\s^{n-1}\!+x)=\frac{1}{\vol_{n-1}(\s^{n-1})}\!\int_{\s^{n-1}}\!\varphi(x+rz)dz\to \varphi(x)\;\hbox{when }r\!\to0.
    \label{eq:lebdiff}
    \end{equation}
 This means that $\phi$ is constant. The claim follows.
    \end{proof}

    \begin{rem}
    \label{r:fabian}
        By using \cite[Lemma 3.1.1]{Sch14:book}, one deduces that the functional $K \mapsto \mu^+(\partial K)$ is a valuation on the set of convex bodies for a Radon measure $\mu$. If one knows that $\mu$ has a density, then one can easily show it is a continuous valuation. In the proof of Theorem~\ref{t:surarecom}, we showed that a Radon measure with density that satisfies the hypothesis of the theorem is translation invariant.  Consequently, one can make the use of Lebesgue differentiation in \eqref{eq:lebdiff} more robust with the following argument from the theory of valuations: first, use McMullen's decomposition \cite[Theorem 6.3.5]{Sch14:book} to write $\mu^+(\partial K)=\sum_{j=1}^nM_j(K)$, where $M_j$ is a continuous, translation invariant, $j$-homogeneous valuation. Then, argue that the support of $\mu^+$ is $(n-1)$-dimensional, and thus $\mu^+ = M_{n-1}$. Lebesgue's differentiation theorem then yields $\phi(x) = \lim_{r\to 0^+} \frac{\mu^+(\partial rK+x)}{r^{n-1}\vol_{n-1}(\partial K)} = \frac{M_{n-1}(K)}{\vol_{n-1}(\partial K)}$, for almost all $x$.
    \end{rem}

    \begin{rem}
        Notice that Theorem~\ref{t:Radon} can also be proven as a corollary of Theorem~\ref{t:surarecom}.
    \end{rem}

One may have observed that, in the proof of Theorem~\ref{t:surarecom}, $L$ was not constrained to contain the origin, and thus $K+L$ may not necessarily contain $K$.  We ponder if we can extend Lemma~\ref{l:dim_1}, i.e. require the assumption that $L$ contains the origin, to higher dimensions and recover the classification of the Lebesgue measure from Theorem~\ref{t:surarecom}.
\begin{conj}
    \label{ques:strict}
        For $n\geq 2$, let $\mu$ be a Radon measure on $\R^n$ with the following property: for every convex body $K$ and compact, convex set $L$ such that $L$ contains the origin, 
        \begin{equation}\mu^+(\partial (K+L)) \geq \mu^+(\partial K).
        \label{eq:wsaa}
        \end{equation}
       Then $\mu$ is a constant multiple of the Lebesgue measure.
    \end{conj}

We now weaken the assumptions further by assuming that $L$, while still containing the origin, is symmetric about a point. We obtain in Theorem~\ref{t:contr} that, in this case, a characterization of the Lebesgue measure like the one shown in Theorem~\ref{t:surarecom} does not hold. We will focus on the planar case, and make use of the fact that every symmetric planar convex body is a zonoid. A \textit{centered zonotope} $Z$ is the Minkowski sum of symmetric line-segments, i.e. it can be written in the form
\begin{equation}
    Z=\sum_{i=1}^ma_i[-u_i,u_i], \; u_i\in\s^{n-1}, \; a_i\in(0,\infty),
    \label{eq:zon}
\end{equation}
where $\s^{n-1}$ is the unit sphere in $\R^n.$ Furthermore, a \textit{centered zonoid} is the limit, with respect to the Hausdorff metric, of a sequence of centered zonotopes; $\mathcal{Z}^n$ denotes the set of centered zonoids in $\R^n.$ A zonoid (resp. zonotope) is merely a translation of a centered zonoid (resp. zonotope). 

To prove Theorem~\ref{t:contr}, we first consider the case when $L$ is a line segment. We establish that, when $\mu$ has density $\phi((x_1,x_2))=x_1^2+x_2^2$, then the property holds.

    \begin{prop}
    \label{p:counter}
        Let $K$ be a convex body in $\R^2$. Let $\mu$ be the Borel measure with density $\phi(x)=|x|^2$. Then, for every $u\in\R^2$,
        $$\mu^+(\partial(K+[0,u])) \geq \mu^+(\partial K).$$
    \end{prop}
    \begin{proof}
        Clearly, the inequality is true for $u$ the origin, and so we suppose $|u|>0$. From rotation invariance, we may assume $u=(0,t)=te_2$, $t>0$. In the same spirit as in Proposition~\ref{p:sup_deriv_equiv}, we localize the result. We compute
        $$\mu(K;B_2^n,[0,e_2])=\lim_{t\to 0^+} \frac{\mu^+(\partial(K+[0,te_2])) - \mu^+(\partial K)}{t},$$ i.e. we will show $\mu(K;B_2^n,[0,e_2]) \geq 0$ for every planar convex body, which is equivalent to our claim via a localization argument. Observe that by denoting by $\mathcal{H}^1$ the $1$-dimensional Hausdorff measure, we can write
        \begin{equation}\frac{\mu^+(\partial(K+[0,te_2])) - \mu^+(\partial K)}{t} = \frac{1}{t}\left( \!\int_{\partial(K+[0,te_2])}\phi(x)d\mathcal{H}^{1}(x)-\!\int_{\partial K}\phi(x)d\mathcal{H}^{1}(x)\right).
        \label{eq:boundary_int_diff}
        \end{equation}
        
        Next, we notice that $K=\{x\in\R^2: x_1\in[a,b], g(x_1) \leq x_2 \leq f(x_1)\}$, where $f$ is a one-dimensional concave function, $g$ is a one-dimensional convex function, both defined on $[a,b] \subset \R$. This means that $K+[0,te_2]=\{x\in\R^2: x_1\in[a,b], g(x_1) \leq x_2 \leq f(x_1)+t\}$. With this parameterization, observe that $\partial K$ can be written as
        $$
        \{(x,f(x)): x\in[a,b]\}\bigcup \{(x,g(x)): x\in [a,b]\}\bigcup (a \times [g(a),f(a)]) \bigcup (b\times [g(b),f(b)]),$$
        and thus $\partial (K+[0,te_2])$ can be represented as
        $$\{(x,f(x)+t): x\in[a,b]\}\bigcup \{(x,g(x)): x\in [a,b]\}\bigcup (a \times [g(a),f(a)+t]) \bigcup (b\times [g(b),f(b)+t]).
        $$
        Consequently,  \eqref{eq:boundary_int_diff} can be written as $\frac{1}{t}(\delta_f(a)+\delta_f(b)+L(f+t)-L(f))$, where
        $$\delta_f(a)= \int_{f(a)}^{f(a)+t} \phi(a,z) dz \mbox{ and  } 
        \delta_f(b)= \int_{f(b)}^{f(b)+t} \phi(b,z) dz,
       $$
       and $L(f)$ is the arc-length of $f$ over $[a,b]$ averaged with respect to $\phi$, i.e. $$L(f)=\int_a^b\phi(z,f(z))\sqrt{1+(f^\prime(z))^2}dz \mbox { and 
 } L(f+t)=\int_a^b\phi(z,f(z)+t)\sqrt{1+(f^\prime(z))^2}dz.$$
        Taking limits, we obtain $$\lim_{t\to0^+}\frac{\delta_f(a)}{t}=\phi(a,f(a))=a^2+f(a)^2$$
        from the Lebesgue differentiation theorem, and similarly for $\delta_f(b)$. One obtains from dominated convergence that
        \begin{align*}\lim_{t\to 0^+}\frac{L(f+t)-L(f)}{t}&=\int_a^b \pdv{}{y}\phi(z,f(z))\sqrt{1+(f^\prime(z))^2}dz
        \\
        &=2\int_a^b f(z)\sqrt{1+(f^\prime(z))^2}dz.\end{align*}
        Putting all the pieces together, we must show for every concave function $f:[a,b]\to \R$ the following inequality:
        $$a^2+b^2+f(a)^2+f(b)^2 + 2 \int_a^b f(t)\sqrt{1+(f^\prime(t))^2}dt \geq 0.$$
        We first observe that, if we establish the inequality for $$f_-(z)=\begin{cases}
      0 & f(z) \geq 0\\
      f(z) & f(z) <0,\end{cases}$$
      then we establish the inequality for $f$. Hence, we may suppose $f$ is non-positive. We then let $h=-f$. Consequently, our goal is establish the following: given $[a,b]\subset \R$, we must show for every non-negative function $h$ convex over $[a,b]$ that
      \begin{equation}
      \label{eq:convex_inequality}a^2+b^2+h(a)^2+h(b)^2 \geq 2 \int_a^b h(z)\sqrt{1+(h^\prime(z))^2}dz,\end{equation}
      which is established in Appendix B.
      \end{proof}
By iterating Proposition~\ref{p:counter} and using monotone convergence, we obtain the following.
    \begin{thm}\label{t:contr}
    Let $K$ be a convex body in $\R^2$ and let $\mu$ be the measure with density $|(x_1,x_2)|^2=x_1^2+x_2^2$. Then, for every compact, convex set $L\subset \R^2$ that contains the origin and is symmetric about a point:
        $$\mu^+(\partial(K+L)) \geq \mu^+(\partial K).$$
    \end{thm}
    To justify this iteration, we first note that $\mu^+(\partial(K+L))$ is continuous in the variable $L$ (with respect to the Hausdorff metric) over the class of zonoids. Indeed, Proposition~\ref{prop:arb_mixed} reduces the question to continuity of $S^\mu_{K+L}$ in the variable $L$, which was shown in \cite[the proof of Theorem 2.7]{FLMZ24:1}. We next note that every compact, convex set $L$ that is symmetric about a point and contains the origin, e.g. zonoids containing the origin, can be approximated by zonotopes containing the origin. Thus, to show that the set of convex bodies generated by Minkowski sums of the form $[0,v]$ for $v\in\R^n$ coincides to the set of zonoids containing the origin, it suffices to show only the case of zonotopes.
    \begin{prop}
    \label{p:zon_ori}
        A set $Z\subset\R^n$, $n\geq 1$, is a zonotope containing the origin if and only if it can be written as $Z=\sum_{i=1}^M [0,v_i]$ for a sequence $\{v_i\}_{i=1}^M\subset\R^n\setminus\{0\}$.
    \end{prop}
    \begin{proof}
    Let $I_i, i=1,\dots,M$ be line segments in $\R^n.$ Then, $Z=\sum_{i=1}^MI_i$ contains the origin if and only if there exist $x_i\in I_i$ such that $\sum_{i=1}^Mx_i=0.$ Thus, we have
    \begin{align*}
        Z=\sum_{i=1}^MI_i =\sum_{i=1}^MI_i - \sum_{i=1}^Mx_i = \sum_{i=1}^M(I_i-x_i).
    \end{align*}
    Observe that each $I_i-x_i$ is now a line segment containing the origin. Moreover, consider a line segment $[a,b]$ containing the origin. Then $[a,b]=[a,0]+[0,b]$, thus $[a,b]$ is the Minkowski sum of two line segments containing the origin as endpoints.
    \end{proof}

    Dominique Malicet asked if it was possible to strengthen Theorem~\ref{t:contr} by replacing $B_2^n$ in $\mu^+(\partial A)=\mu(A;B_2^n)$ with an arbitrary symmetric convex body. In Theorem~\ref{thm:strict_2} below, we show that this is not true. In fact, we show that a Radon measure satisfying these requirements must again be the Lebesgue measure. In the Appendix A, Section~\ref{app:rmk}, we recall the prequel to this work, \cite{FLMZ24:1}. In that work, we had found a formula for $\mu(A;B,C)$, which we recall in Theorem~\ref{t:best_thm}. In the following proposition, we derive a particular case of $\mu(A;B,C)$ which will be needed in the proof of Theorem~\ref{thm:strict_2}.
\begin{prop}
    Let $A$ be a convex body and $C$ and a compact, convex set, both in $\R^n$. Suppose $\mu$ is a Borel measure on $\R^n$ with $C^1$ density $\varphi$. Then, for every $v\in\R^n\setminus\{0\}$, the following holds.  If $n\geq 3$:
    \begin{equation}
    \label{eq:sym_seg_0}
     \begin{split}
        \frac{1}{|v|}\mu(A;[0,v],C) =(n-1)&\int_{\{y\in \partial A: \langle n_A(y),v\rangle=0\} }h_C(n_A(y))\phi(y)dy 
        \\
        &+\int_{\{y\in \partial A:\langle n_A(y),v \rangle \geq 0\}}h_C(n_A(y))\left\langle\nabla \phi(y),\frac{v}{|v|}\right\rangle dy.
    \end{split}
    \end{equation}
    and, if $n=2$:
    \begin{equation}
    \label{eq:sym_seg}
    \begin{split}
        \frac{1}{|v|}\mu(A;[0,v],C) = h_C(-v^\perp)&\mu^+(\{y\in \partial A: n_A(y)=-v^\perp\}) 
        \\
        &+ h_C(v^\perp)\mu^+(\{y\in \partial A: n_A(y)=v^\perp\})
        \\
        &\quad +\int_{\{y\in \partial A:\langle n_A(y),v \rangle \geq 0\}}h_C(n_A(y))\left\langle\nabla \phi(y),\frac{v}{|v|} \right\rangle dy,
    \end{split}
    \end{equation}
    where $\pm v^\perp$ denote the two unit vectors in $\s^1$ that are orthogonal to $v$.
    \label{p:sym_seg}
\end{prop}
\begin{proof}
We first consider the case when $n\geq 3$. Recall that the mixed surface-area measure $dS_{K,L}$ is $1$-homogeneous in the argument of the compact, convex set $L$. Denote $\hat v=v/|v|$ and observe that $dS_{A[n-2],[0,v]}=|v|dS_{A[n-2],[0,\hat v]}$ and that $dS_{A[n-2],[0,\hat v]}$ is the restriction of $dS_A$ to $\hat v^\perp$. Suppose that $A$ is $C^2_+$. Then, these observations and \eqref{eq:deriv} yield that $\mu(A;[0,v],C)$ equals
    \begin{align*}
\int_{\s^{n-1}}h_C(u)\left((n-1)\phi(n^{-1}_A(u))|v|\chi_{\{\langle u,v \rangle=0\}}(u) + \langle\nabla \phi(n^{-1}_A(u)),\nabla h_{[0,v]}(u)\rangle \right)dS_A(u).
    \end{align*}
    Writing the above over $\partial A$, this becomes
      \begin{equation*}
\int_{\partial A}h_C(n_A(y))\left((n-1)\phi(y)|v|\chi_{\{\langle n_A(y),v \rangle=0\}}(y) + \langle\nabla \phi(y),\nabla h_{[0,v]}(n_A(y))\rangle \right)dy.
    \end{equation*}
We can then drop the assumption that $A$ is $C^2_+$ with an approximation argument. Next, by using that $\nabla h_{[0,v]}(n_A(y))=v\cdot \chi_{\{y\in \partial A:\langle n_A(y),v \rangle \geq 0\}}(y)$ a.e, we obtain \eqref{eq:sym_seg_0}.

     For the case $n=2$, we observe that $dS_{A[n-2],[0,v]}=dS_{[0,v]}$. Then, using the fact that surface area in the plane is $1$-homogeneous in the argument of the convex set, we obtain $dS_{[0,v]}=|v|dS_{[0,\hat v]}$. It is well-known that, for $\theta\in\s^{n-1}$, $dS_{[0,\theta]}$ is the restriction of the spherical Lebesgue measure to $\s^{n-1}\cap \theta^\perp$. In the case of a planar unit vector, this is just Dirac masses at the two points on the sphere orthogonal to $\theta$, which in this case are precisely the two points we denote by $\pm v^\perp$.
    Thus, we obtain
    $$dS_{A[n-2],[0,v]}(u)=|v|dS_{[0,\hat v]}(u)=|v|\left(\delta_{-v^\perp}(u)+\delta_{v^\perp}(u)\right).$$
       If, moreover, $A$ is $C^2_+$, then, inserting the above observations into \eqref{eq:deriv}, we obtain 
    \begin{align*}
\int_{\s^{1}}h_C(u)\left(\phi(n^{-1}_A(u))|v|\left(\delta_{-v^\perp}(u)+\delta_{v^\perp}(u)\right) + \langle\nabla \phi(n^{-1}_A(u)),\nabla h_{[0,v]}(u)\rangle dS_A(u)\right).
    \end{align*}
    Writing the above over $\partial A$, this becomes
      \begin{equation*}
\int_{\partial A}h_C(n_A(y))\left(\phi(y)|v|\left(\delta_{-v^\perp}(n_A(y))+\delta_{v^\perp}(n_A(y))\right) + \langle\nabla \phi(y),\nabla h_{[0,v]}(n_A(y))\rangle dy\right).
    \end{equation*}
    At this point, we can drop the assumption that $A$ is $C_2^+$. By taking limits, this becomes:
     \begin{equation*}
    \begin{split}
        \mu(A;B,C) &=|v|\left( h_C(-v^\perp)\int_{\{y\in \partial A: n_A(y)=-v^\perp\} }\phi(y)dy + h_C(v^\perp)\int_{\{y\in \partial A: n_A(y)=v^\perp\} }\phi(y)dy\right)
        \\
        &+\int_{\partial A}h_C(n_A(y))\langle\nabla \phi(y),\nabla h_{[0,v]}(n_A(y))\rangle dy.
    \end{split}
    \end{equation*}
Next, by using that $\nabla h_{[0,v]}(n_A(y))=v\cdot \chi_{\{y\in \partial A:\langle n_A(y),v \rangle \geq 0\}}(y)$ a.e, we obtain
 \begin{equation*}
    \begin{split}
        \mu(A;B,C) &=|v|\left( h_C(-v^\perp)\int_{\{y\in \partial A: n_A(y)=-v^\perp\} }\phi(y)dy + h_C(v^\perp)\int_{\{y\in \partial A: n_A(y)=v^\perp\} }\phi(y)dy\right)
        \\
        &+\int_{\{y\in \partial A:\langle n_A(y),v \rangle \geq 0\}}h_C(n_A(y))\langle\nabla \phi(y),v\rangle dy.
    \end{split}
    \end{equation*}
    Inserting the definition of $\mu^+$ and dividing by $|v|$ yields \eqref{eq:sym_seg}.
\end{proof}

    \begin{thm}
    \label{thm:strict_2}
       Let $n\geq 1$, and suppose $\mu$ is a Radon measure on $\R^n$ such that: for every zonoid $A$, every centered zonoid $B$ and every zonoid $C$ containing the origin, one has
        $$\mu(A+C;B) \geq \mu(A;B).$$
        Then, $\mu$ is a constant multiple of the Lebesgue measure.
    \end{thm}

The $n=1$ case of Theorem~\ref{thm:strict_2} is precisely again Lemma~\ref{l:dim_1}. For higher dimensions, the theorem will actually be a consequence of a stronger result, Proposition~\ref{p:strict_2} below. We first show how the proposition implies the theorem, and then prove the proposition.

\begin{prop}
\label{p:strict_2}
    Let $n\geq 2$, and suppose $\mu$ is a Radon measure on $\R^n$ with $C^1$ density such that: for every zonoid $A$ and vector $v\in\R^n\setminus\{0\}$, one has
    $$\mu(A;[-v,v],[0,v]) \geq 0.$$
    Then, $\mu$ is a multiple of the Lebesgue measure.
\end{prop}
\begin{proof}[Proof of Theorem~\ref{thm:strict_2}]
    By using Proposition~\ref{p:sup_deriv_equiv}, with $\mathcal{A}$ all zonoids, $\mathcal{B}$ all centered zonoids and $\mathcal{C}$ all zonoids that contain the origin, this is equivalent to $\mu(A;B,C) \geq 0$, where $A\in\mathcal{A}, B\in\mathcal{B}$ and $C\in\mathcal{C}$. From \eqref{mixed_multi}, this statement is equivalent to $\mu(A;[-z,z],[0,v]) \geq 0$ holding for every $z,v\in\R^2\setminus\{0\}$. Notice we can again assume via Proposition~\ref{p:sup_mod_shift} that $\mu$ has a density that is continuous with respect to the Lebesgue measure. We then consider the simpler case when $z=v$, which is precisely the content of Proposition~\ref{p:strict_2}.
\end{proof}

    \begin{proof}[Proof of Proposition~\ref{p:strict_2}] Let $\mu$ have density $\varphi$. Consider $v\in \s^{n-1}$, and for any $x\in \R^n$, consider $(n-1)$-dimensional Euclidean disks $D_v(r,x)$ with outer unit normal $v$, with $r>0$ being the radius of the disk and $x$ being the center of the disk, i.e. $D_v(r,x)=x+(rB_2^n\cap v^\perp)$.   We first observe that Proposition~\ref{p:sym_seg} yielding $\mu(D_v(r,x);[-v,v],[0,v]) \geq 0$ precisely means
\begin{equation}
    \label{eq:sym_seg_new}
    \int_{D_v(r,x)}\left\langle\nabla \phi(y),v\right\rangle dy \geq 0.
    \end{equation}
Next, we obtain from $\mu(D_v(r,x);[-v,v],[0,-v]) \geq 0$ similarly that
\begin{equation}
    \label{eq:sym_seg_new_2}
    \int_{D_{v}(r,x)}\left\langle\nabla \phi(y),v\right\rangle dy \leq 0.
    \end{equation}
Combining the two, we obtain
    \begin{equation}
    \label{eq:sym_seg_new_4}
    \int_{D_{v} (r,x)}\left\langle\nabla \phi(y),v\right\rangle dy = 0,
    \end{equation}
    i.e. the vector field on $\R^n$ given by $$x \mapsto \int_{D_v(r,x)}\nabla \phi(y) dy$$
    is orthogonal to $v$. Sending $r\to 0$ and using Lebesgue's differentiation theorem in the affine subspace of $\R^n$ containing $D_v(r,x)$, we deduce that $\nabla \phi(x)$ is orthogonal to $v$.

    We then vary $v$ over, say, the canonical basis vectors to deduce $\nabla \phi(x)$ is the origin. Since this is true for every $x\in\R^n$, we must have that $\phi$ is constant.
      \end{proof}

 \begin{rem}
        Notice that Theorem~\ref{t:Radon} can also be proven as a corollary of Theorem~\ref{thm:strict_2}.
    \end{rem}

\subsection{Direct characterizations of super-or-sub-modularity}
From the results of Sections~\ref{sec:lebclass} and \ref{sec:weaksup}, it seems reasonable to only consider super/submodularity over symmetric convex sets and rotational invariant measures. For example, we establish in this section that, even when considering only symmetric convex sets, the Gaussian measure is neither submodular nor supermodular, except in dimension $1$. We first restrict our sets to dilates of a fixed Borel set.

\begin{lem}
\label{l:condition}
Let $\mu$ be a Borel measure on $\R^n$, $n\geq 1$, and fix a convex set $K\subset \R^n$.   Then, 
\begin{equation}\label{eq:conv}
\mu(aK+bK+cK) + \mu(aK) \geq \mu(aK+bK) + \mu(aK+cK), \mbox{  } \forall a,b,c \ge 0
\end{equation}
if and only if the function $t\mapsto\mu(tK)$ is a convex function for $t>0$. Similarly, 
$$
\mu(aK+bK+cK) + \mu(aK) \leq \mu(aK+bK) + \mu(aK+cK), \mbox{  } \forall a,b,c \ge 0
$$
if and only if the function $t\mapsto\mu(tK)$ is a concave function for $t>0$.
\end{lem}

\begin{proof}
The proof follows immediately by rewriting (\ref{eq:conv}) as
$$\frac{\mu((a+b+c)K)-\mu((a+b)K)}{c}\geq \frac{\mu((a+c)K)-\mu(aK)}{c},$$
which is an equivalent definition of a function being convex. The second inequality follows similarly.
\end{proof}
\noindent We henceforth consider measures $\mu$ with density $V(|x|)$ and study the restrictions  on $V$ for supermodularity of $\mu$ to hold. In the following proposition, we elaborate on Lemma~\ref{l:condition} in the case when $K$ is a centered Euclidean ball and $\mu$ is a rotational invariant measure. 
\begin{prop}
\label{p:dim_n}
Suppose $\mu$ is a rotational invariant measure on $\R^n$, $n\geq 1$, of the form $d\mu(x)=V(|x|)dx,$ $V:[0,\infty)\to\R_+$. Then
$$
\mu(aB_2^n+bB_2^n+cB_2^n) + \mu(aB_2^n) \geq \mu(aB_2^n+bB_2^n) + \mu(aB_2^n+cB_2^n), \mbox{ } 
\forall a,b,c \ge 0$$
if and only if the function
$r\mapsto V(r)r^{n-1}$ is increasing for all $r>0$ and 
$$
\mu(aB_2^n+bB_2^n+cB_2^n) + \mu(aB_2^n) \leq \mu(aB_2^n+bB_2^n) + \mu(aB_2^n+cB_2^n), \mbox{ } 
\forall a,b,c \ge 0$$
if and only if the function
$r\mapsto V(r)r^{n-1}$ is decreasing for all $r>0$.
\end{prop}

\begin{proof}
Begin by defining the following function by passing to polar coordinates: $$f(r)=\mu(rB_2^n)=\int_{\s^{n-1}}\int_{0}^r V(t)t^{n-1}dtd\theta=\vol_{n-1}(\s^{n-1}) \int_{0}^r V(t)t^{n-1}dt.$$
Taking a derivative, we get $f^\prime(r)=\vol_{n-1}(\s^{n-1}) V(r)r^{n-1}.$
Therefore, we see that $V(r)r^{n-1}$ must be an increasing (decreasing) function for $f$ to be convex (concave), which yields the claim via Lemma~\ref{l:condition}.
\end{proof}
\noindent We obtain the following immediate corollary in the case of dimension $n=1$.
\begin{cor}
\label{cor:dim_1_sub_log}
Let $d\mu(x)=V(x)dx,$ be a measure on $\R$ where $V$ is an even function. Then, for $A,B,C$ symmetric intervals, one has 
$$
\mu(A+B+C) + \mu(A) \geq \mu(A+B) + \mu(A+C)
$$
if and only if
$V$ is non-decreasing (e.g. increasing) on $\R^+$ and
$$
\mu(A+B+C) + \mu(A) \leq \mu(A+B) + \mu(A+C),
$$
if and only if
$V(x)$ is non-increasing (e.g. decreasing) on $\R^+$.
\end{cor}

We conclude this subsection by applying the above results to the Gaussian measure.
\begin{prop}\label{p:mod_log}
The measure $\gamma_1$ is submodular on the class of symmetric convex bodies. That is, for symmetric intervals $A,B,C\subset\R,$ one has
$$\gamma_1(A+B+C) + \gamma_1(A) \leq \gamma_1(A+B) + \gamma_1(A+C).$$
For $n\geq 2,$ $\gamma_n$ is neither submodular nor supermodular over the class of symmetric convex sets.
\end{prop}
\begin{proof}
From Corollary~\ref{cor:dim_1_sub_log}, $\gamma_1$ is submodular. For $n\geq 2,$ the function $r\mapsto\frac{1}{(2\pi)^{n/2}}e^{-\frac{r^2}{2}}r^{n-1}$ is increasing for $r\in[0,\sqrt{n-1})$ and  decreasing  for $r\in[\sqrt{n-1},\infty)$. Therefore, from Proposition~\ref{p:dim_n}, it follows that $\gamma_n$ is neither submodular nor supermodular.
\end{proof}

 \subsection{Log-submodularity}
\label{sec:logsub}
In this subsection, we use the theory of mixed measures to study the following three body inequality in the case of other $\log$-concave measures instead of volume: let $\mu$ be a log-concave measure supported on a class of compact sets $\mathcal{C}$, then we wish to determine the best constant $c_{\mu,n}$ (for this class) such that for $A,B,C\in \mathcal{C}$
 \begin{equation}
 \label{eq:question}
 \mu(A)\mu(A+B+C) \le c_{\mu,n} \mu(A+B)\mu(A+C).
 \end{equation}
Inspired by the known submodularity of entropy for convolutions \cite{Mad08:itw, MK18} and its interesting consequences for convex geometry \cite{BM11:cras}, Bobkov and the third named author had first studied this concept in the context of the Lebesgue measure; they showed \cite{BM12:jfa} that for convex bodies $A,B$ and $C$ that
$$\vol_n(A)\vol_n(A+B+C) \leq 3^n  \vol_n(A+B)\vol_n(A+B).$$
Defining, for compact, convex sets $A,B,$ and $C$
\begin{equation}
c(A,B,C)=\frac{\vol_n(A)\vol_n(A+B+C)}{\vol_n(A+B)\vol_n(A+B)} \quad   \mbox{   and    }  c_n=\sup_{\substack{\text{compact,} \\ \text{convex sets} \\ A,B,C\subset \R^n}} c(A,B,C),
\label{eq:BMcons}
\end{equation}
the result of \cite{BM12:jfa} shows $c_n \leq 3^n.$ It was shown in \cite{FMZ24} that $c_2=1$, $c_3=4/3$ and, for $n\geq 4,$ $(4/3 +o(1))^{n}\le c_n\le  \varphi^{n}$, where $\varphi=(1+\sqrt{5})/2$ is the golden ratio.
 
 We first obtain the following immediate result from the definition of log-concavity. 
 \begin{prop}[Log-submodularity for dilated sets]
 \label{p:a}
 Let $\mu$ be a log-concave Borel measure on $\R^n$, $n\geq 1$. Fix convex sets $A$ and $B$ in $\R^n$. Then, if $C=tB$, for some $t>0$, one has
 $$\mu(A)\mu(A+B+C) \le \mu(A+B)\mu(A+C).$$
 \end{prop}
 \begin{proof}
 Observe that
 $$
 \mu(A+tB) =\mu \left(\frac{1}{t+1}  A + \frac{t}{t+1} ( A+(t+1)B)\right)\ge \mu\left(A\right)^{\frac{1}{t+1}} \mu\left( A+(t+1)B\right)^{\frac{t}{t+1}}
 $$
 and
 $$
 \mu(A+B) =\mu \left(\frac{t}{t+1}  A + \frac{1}{t+1} ( A+(t+1)B)\right)\ge \mu\left(A\right)^{\frac{t}{t+1}} \mu\left( A+(t+1)B\right)^{\frac{1}{t+1}}.
 $$
 \end{proof}
 
\begin{rem}Convexity in the proof of the above proposition was essential, as it allowed the distribution of dilation over Minkowski summation (e.g. $(t+1)B=tB+B$). In fact, convexity is required in general. Indeed, let $A=[-1,1],$ $B=C=[-r,-r+1]\cup [r-1,r]$ for $r>2.$ Notice that $B+C=[-1,1]\cup[-2r,-2r+2]\cup[2r-2,2r]$ and $A+B=B+C=[-r-1,-r+2]\cup[r-2,r+1].$ Thus, for any (strictly) radially decreasing log-concave measure on $\R,$ e.g. $\gamma_1,$ one has $\gamma_1(A)\gamma_1(A+B+C) > \gamma_1(A)\gamma_1(B+C) > \gamma_1([-1,1])^2 >0$ and yet $\gamma_1(A+B)\gamma_1(A+C)$ goes to zero as $r \to \infty.$ 
\end{rem}

Despite the positive result from Proposition~\ref{p:a}, the inequality \eqref{eq:question} cannot hold in general if $\mathcal{C}$ is not a class of  convex sets containing the origin. Denoting the unit Euclidean ball as $B_2^n,$ let $B=B_2^n+rv$ and $C=B_2^n-rv,$ where $v\in \s^{n-1}$ and $r>0$. Then, the left-hand side of \eqref{eq:question} is independent of $r,$ but the right-hand side is a function of $r$. These two examples show, in fact, that the weaker inequality
$$\mu(A)\mu(B+C)\leq \mu(A+B)\mu(A+C)$$
can only possibly be true for all log-concave measures $\mu$ if $\mathcal{C}$ is a class of convex sets containing the origin. This is surprising, because Ruzsa \cite{Ruz97} showed that
$$\vol_n(A)\vol_n(B+C)\leq \vol_n(A+B)\vol_n(A+C)$$
for every compact $A,B,C\subset\R^n.$

Notice that every symmetric convex interval in $\R$ is a dilate of $[-1,1]$, and so Proposition~\ref{p:a} yields that every log-concave measure on $\R$ satisfies \eqref{eq:question}, with $c_{\mu,1}=1$ and $B$ and $C$ symmetric intervals. In particular, the Gaussian measure. We collect this observation below.
 \begin{thm} {(The case of log-concave measures in $\R$)}
 \label{t:gauss_r}
Let $\mu$ be a log-concave measure on $\R$. Then, for any interval $A\subset\R$ and any symmetric intervals $B, C \subset \R$, one has 
 $$
 \mu(A)\mu(A+B+C) \le  \mu(A+B)\mu(A+C).
 $$
 \end{thm}
 
In the volume case, it is known that \eqref{eq:question} with $c_{\mu,n}=1$ for $n\geq 3$ is false. We leave determining which measures and which classes of convex sets satisfy \eqref{eq:question} with $c_{\mu,n}=1$ as an open question. That is, if $\mathcal{C}$ is some fixed class of convex sets, we ask for which measures $\mu$ does it hold that, for $A,B,C\in\mathcal{C}$, one has 
\begin{equation}
\label{eq:good_ineq}
\mu(A)\mu(A+B+C) \le  \mu(A+B)\mu(A+C).\end{equation}
We conjecture that \eqref{eq:good_ineq} holds for $A,B,C$ being symmetric zonoids and $\mu$ being any log-concave measure. In the volume case, log-submodularity with constant $1$ was established recently by a subset of the authors, working with Meyer, in $\R^2$ and $\R^3$ \cite{FMMZ24} when $A,B$ and $C$ are taken to be zonoids; the $2$-dimensional case follows from \cite{SZ16}.

We conclude this subsection by showing that \eqref{eq:good_ineq} is related to the so-called B\'ezout-type inequalities for mixed measures (see \cite{FLMZ24:1}).

\begin{lem}[Local form of log-submodularity]
\label{l:m}
Fix $n\geq 1$. Let $\mu$ be a Radon measure on $\R^n$ and $\mathcal{C}$ a class of convex sets, with the property that $\mu(A;B,C)$ and $\mu(A;B)$ exist and are finite for every $A,B,C\in\mathcal{C}$. Then, \eqref{eq:good_ineq} is true for $\mu$ and for every $A,B,C\in\mathcal{C}$ if and only if
\begin{equation}\mu(A)\mu(A;B,C) \leq \mu(A;B)\mu(A;C).
\label{eq:inequality_second}
\end{equation}
\end{lem}
\begin{proof}
We see that \eqref{eq:good_ineq} is equivalent to, for all $B,C\in\mathcal{C}$ and $s,t\ge0$, that
\begin{equation}\label{eq:mu:abc}
\mu(A)\mu(A+sB+tC) \le \mu(A+sB)\mu(A+tC).
\end{equation}
Let $g(s,t)=g_{A,B,C}(s,t)=\mu(A+sB+tC)$. Then inequality \eqref{eq:mu:abc} may be rewritten as  $g(s,t)g(0,0)\le g(s,0)g(0,t)$. Let $\eta(s,t)=\eta_{A,B,C}(s,t)=\ln(g(s,t))$. This may also be written as $\eta(s,t)+\eta(0,0)\le \eta(s,0)+\eta(0,t)$. This clearly implies that $\frac{\partial^2\eta_{A,B,C}}{\partial s\partial t}(0,0)\le0$. The reciprocal is true if we want to prove the inequality for a class of Borel sets closed under Minkowski summation. Indeed, for any $s_0,t_0\ge0$, one has
\[
\eta(s_0,t_0)-\eta(s_0,0)-\eta(0,t_0)+\eta(0,0)=\int_0^{s_0}\int_0^{t_0}\frac{\partial^2\eta}{\partial s\partial t}(s,t)dsdt.
\]
Thus, to prove \eqref{eq:mu:abc}, it is enough to prove that, for all $s,t\ge0$, $\frac{\partial^2\eta}{\partial s\partial t}(s,t)\le0$. But, using that 
$\eta_{A,B,C}(s_0+s,t_0+t)=\eta_{A+s_0B+t_0C,B,C}(s,t)$ from convexity, we see that 
\[
\frac{\partial^2\eta_{A,B,C}}{\partial s\partial t}(s_0,t_0)=\frac{\partial^2\eta_{A+s_0B+t_0C,B,C}}{\partial s \partial t}(0,0).
\]
Therefore, the inequality \eqref{eq:mu:abc} holds for all $A,B,C$ in a class of sets closed by Minkowski sums if and only if for all $A,B,C$ in this class $\frac{\partial^2\eta_{A,B,C}}{\partial s\partial t}(0,0)\le0$. Next, observe that $\frac{\partial^2\eta_{A,B,C}}{\partial s\partial t}(0,0)\le0$ is true if and only if
\begin{equation}\left(\pdv{g}{s,t}(0,0)\right)g(0,0)\le \pdv{g}{s}(0,0)\pdv{g}{t}(0,0).
\label{eq:mixed_deriv}
\end{equation}
We see that $$\pdv{g}{s}(0,0)=\mu(A;B), \pdv{g}{t}(0,0)=\mu(A;C), \; \text {and } \pdv{g}{s,t}(0,0)=\pdv{\mu(A+sB;C)}{s}(0).$$
Therefore, \eqref{eq:mixed_deriv} can be written as
\begin{equation*}\mu(A;B)\mu(A;C)\geq \mu(A)\left(\pdv{\mu(A+sB;C)}{s}(0)\right).
\end{equation*}
And then the result follows from \eqref{eq:limit_def_alt}.
\end{proof}

This global-to-local approach for the question of log-submodularity was first done in the volume case. In this instance, \eqref{eq:inequality_second} becomes

\begin{equation}
V(A[n-1],B)V(A[n-1],C)\geq \frac{n-1}{n}\vol_n(A)V(A[n-2],B,C).
\label{eq:AFZ}
\end{equation}
Therefore, \eqref{eq:AFZ} is referred to as \textit{local log-submodularity}. The inequality in \eqref{eq:AFZ} was verified to hold for all planar convex bodies \cite{SZ16} (and this is precisely equivalent to $c_2=1$ in \eqref{eq:BMcons} above). The equivalence between log-submodularity with constant $1$ and local log-submodularity therefore means that the work from \cite{FMMZ24} mentioned after \eqref{eq:zon} for zonoids in $\R^3$ is precisely equivalent to \eqref{eq:AFZ} holding for $A,B,C\in\mathcal{Z}^3.$  

Sometimes, it is possible to establish \eqref{eq:inequality_second} for a fixed set $A$ (indeed, in the prequel \cite{FLMZ24:1}, we considered the case $A=RB_2^n$ for $R>0$). Unfortunately, this does not establish \eqref{eq:good_ineq} for this $A.$ Indeed, if one were to fix $A$ and repeat the proof of Lemma~\ref{l:m}, then using that
\[
\frac{\partial^2\eta_{A,B,C}}{\partial s\partial t}(s_0,t_0)=\frac{\partial^2\eta_{A+s_0B+t_0C,B,C}}{\partial s \partial t}(0,0),
\]
the local form of log-submodularity would be to show, for every $s_0,t_0>0$ and $B,C\in\mathcal{C}$ that
$$\mu(A+s_0B+t_0C;B)\mu(A+s_0B+t_0C;C)\geq \mu(A+s_0B+t_0C)\mu((A+s_0B+t_0C);B,C).$$
Note that that this is \eqref{eq:inequality_second} not with $A$ but with $A+s_0B+t_0C.$

\section{Appendix A: existence of mixed measures}\label{app:rmk}
In this section, we briefly mention some known facts about $\mu(K;L)$, as well as advertising the study of $\mu(A;B,C)$ in the prequel \cite{FLMZ24:1}. Mainly, the purpose of this is to pacify the reader; that these quantities exist under minor assumptions, and so our results are not vapid.

We first recall the surface area measure of a $K\subset \R^n$ a compact, convex set: 
\begin{equation}
    \label{eq:surface}
    S_{K}(E)=\int_{n_K^{-1}(E)}d\mathcal{H}^{n-1}(y)
\end{equation}
for every Borel set $E \subset \s^{n-1},$ where $\mathcal{H}^{n-1}$ is the $(n-1)$ dimensional Hausdorff measure (henceforth, to be written as $dy$ for brevity). Here, the Gauss map $n_K:\partial K \to \s^{n-1}$ associates an element on the boundary of $K$ with its outer unit normal. The Gauss map is unique for almost all $x\in\partial K.$ 
If a convex body $K$ has positive radii of curvature everywhere, we say $K$ has \textit{positive curvature}. In this instance, there exists a continuous, strictly positive function $f_K(u)$, the \textit{curvature function} of $K$, such that one has $dS_K(u)=f_K(u)du$ and $n_K$ is a bijection between $\partial K$ and $\s^{n-1}$. It is standard to denote $C^2$ smooth convex bodies with positive curvature as being of the class $C^2_+$. From \cite[Theorem 2.7.1]{Sch14:book}, every compact, convex set can be uniformly approximated by convex bodies that are $C^2_+.$

Using the surface area measure, one obtains the following formula for mixed volumes:
\begin{equation}
	    V(K[n-1],L):=\frac{1}{n}\lim_{\epsilon\to0^+}\frac{\vol_n(K+\epsilon L)-\vol_n(K)}{\epsilon}=\frac{1}{n}\int_{\s^{n-1}}h_L(u)dS_K(u),
	    \label{eq:mixed_2}
\end{equation}
where $h_K(x)=\sup_{y\in K}\langle y,x\rangle$ is the support function of $K$. Support functions work well with Minkowski summation: for $\alpha,\beta>0$ and compact, convex sets $K$ and $L$, one has
$$h_{\alpha K +\beta L} = \alpha h_K + \beta h_L$$
pointwise. The next step is to obtain a formula similar to \eqref{eq:mixed_2} for $\mu(K;L).$

We start the case when $L=B_2^n$. The outer-Minkowski content of a Borel set $A\subset\R^n$ with respect to a Borel measure $\mu$ (not necessarily with density) is given by
$$\mu^+(\partial A) = \liminf_{\epsilon\to 0^+}\frac{\mu\left(K+\epsilon B_2^n\right)-\mu(K)}{\epsilon}.$$
This is a classical definition from the theory of metric measure spaces. Let $K$ be a convex body and $\mu$ a Borel measure $\mu$ on $\R^n$ with density $\phi$. Then, $\mu^+(\partial K)$ is called the \textit{weighted surface area} of $K$ and
\begin{equation}\mu^+(\partial K):=\lim_{\epsilon\to 0^+}\frac{\mu\left(K+\epsilon B_2^n\right)-\mu(K)}{\epsilon}=\int_{\partial K}\phi(y)dy,
\label{eq_bd}
\end{equation}
i.e. the $\liminf$ is a limit and an integral formula is known. It was folklore for a long time that the formula held when $\phi$ is continuous (see e.g. the work by Ball concerning the Gaussian measure \cite{Bal93}). It was shown recently \cite{KL23} that the formula \eqref{eq_bd} holds when $\phi$ contains $\partial K$ in its Lebesgue set. With this minor assumption on $\phi$, one can define the \textit{weighted surface area measure}, denoted $S^{\mu}_{K},$ via the Riesz-Markov-Kakutani representation theorem, since, for a continuous $f\in \mathcal{C}(\s^{n-1}),$
$$f\mapsto \int_{\partial K}f(n_K(y))\phi(y)dy$$ 
is a positive linear functional. That is, by definition $S^{\mu}_{K}$ satisfies the following change of variables formula:
$$\int_{\partial K}f(n_K(y))\phi(y)dy=\int_{\s^{n-1}}f(u)S^{\mu}_{K}(u).
$$
With this rigorous definition of weighted surface area available, one can prove (see \cite{KL23}) the following integral representation of mixed measures by setting $f=h_L$ for some compact, convex set $L\subset \R^n$ (see \cite{Liv19} for an alternative proof in the case when the measure has continuous density). 
\begin{prop}[Representation of mixed measures]
\label{prop:arb_mixed}
	Let $L$ be a compact, convex set and $K$ be a convex body both in $\R^n$. Suppose that $\mu$ is a Borel measure on $\R^n$ whose density contains $\partial K$ in its Lebesgue set. Then, 
	\begin{equation}
	    \mu(K;L)=\int_{\s^{n-1}}h_L(u)dS^{\mu}_{K}(u).
	    \label{eq:arb_mixed}
	\end{equation}
\end{prop}
\noindent Notice that Proposition~\ref{prop:arb_mixed} also yields existence. Recalling that $\mu(tK;K)=\odv{}{t}\mu(tK)$ when $K$ is a convex set, we can use \eqref{eq:arb_mixed} in conjunction with \eqref{eq:measure_mixed_relate_0} when $K$ is a convex body to obtain
\begin{equation}
    \mu(K)=\int_0^1\int_{\s^{n-1}}h_K(u)dS^{\mu}_{tK}(u).
    \label{eq:measure_mixed_relate}
\end{equation}

Next, we have the following result for $\mu(A;B,C)$. Recall that $S_{A[n-2],B}$ denotes the mixed surface area measure of a convex body $A$ and a compact, convex set $B$ (see \cite[Section 5.1]{Sch14:book} for a precise definition). We introduced the weighted analog of the mixed surface area measure $S_{A[n-2],B}$, which is denoted $S^\mu_{A;B}$.

\begin{defn}\label{defn:WMSAM}
Let $A$ be a $C^2_+$ convex body and $B$ be an arbitrary compact, convex set in $\R^n$, $n\geq 2$, and
$\mu$ be a Borel measure with $C^1$ density $\phi$. The  weighted mixed surface area measure $S^\mu_{A;B}$ is the signed measure on $\s^{n-1}$ defined by
$$
dS^\mu_{A;B}(u)=\phi(n^{-1}_A(u))dS_{A[n-2],B}(u) + 
\frac{1}{n-1}\langle\nabla \phi(n^{-1}_A(u)),\nabla h_B(u)\rangle dS_{A}(u) .
$$
\end{defn}

\begin{thm}[Representation of mixed measures of three bodies \cite{FLMZ24:1}]
    Let $\mu$ be a Borel measure on $\R^n,$ $n\geq 2,$ with $C^2$ density $\phi$. For a convex body $A$ and compact, convex sets $B$ and $C$, one has
\begin{equation}
\label{eq:deriv}
\begin{split}
\mu&(A;B,C)=(n-1)\int_{\s^{n-1}}h_C(u) dS^\mu_{A;B}(u).
\end{split}
\end{equation}
\label{t:best_thm}
\end{thm}
\noindent This again establishes the existence of $\mu(A;B,C)$ as well under the assumptions of the theorem. The reason $S^\mu_{A;B}$ exists even when $A$ is not $C^2_+$ is with an argument using the Riesz representation theorem (however, for general $A$, there is no elegant formula like that in Definition~\ref{defn:WMSAM}). In fact, one can again use the Riesz representation theorem and approximation by mollifers to drop the regularity assumptions on $\phi$ (and so $S^\mu_{A;B}$ exists for arbitrary convex body $A$, compact, convex set $B$ and Borel measure with density $\mu$).

\noindent As a consequence of the relation between Minkowski summation and support functions, one has the following bilinearity:
\begin{equation}
\label{mixed_multi}
\mu(A;B,\alpha C_1+\beta C_2)=\alpha\mu(A;B,C_1)+\beta\mu(A;B,C_2).
\end{equation}

\section{Appendix B: an inequality for convex functions on $\R$}\label{sec:Naz}

\subsection{Establishing the inequality}

This section is dedicated to establishing the inequality \eqref{eq:convex_inequality}, which we recall here: for any non-negative, convex function  $h$ on $[a,b]$,
$$a^2+b^2+h(a)^2+h(b)^2 \geq 2 \int_a^b h(x)\sqrt{1+(h^\prime(x))^2}dx.$$
The proof was  communicated to us by Fedor Nazarov. Observe that, for every $a,b\in\R$, one has $2\left(\frac{b-a}{2}\right)^2\leq a^2+b^2.$
Consequently, we will establish the following stronger inequality.

\begin{prop}  
\label{prop:fedja-cvx}
Consider a non-negative, convex function $h:[a,b]\to \R$, then
\begin{equation}
      \label{eq:convex_inequality_2}
      2\left(\frac{b-a}{2}\right)^2+h(a)^2+h(b)^2 \geq 2 \int_a^b h(x)\sqrt{1+(h^\prime(x))^2}dx,\end{equation}
with equality if and only if $h(x)=\alpha x + \frac{\sqrt{1+\alpha^2}-\alpha}{2} b - \frac{\sqrt{1+\alpha^2}+\alpha}{2} a$ for some $\alpha \in \R$.
\end{prop}
\begin{proof}
 We first note, that shifting function $h$, together with interval $[a,b]$ we may assume that the minimum of $h$ is attained at $x=0$ and thus $a \le 0$.  For ease, we replace $a$ with $-a$ to keep $a$ positive. That is, it suffices  to show, for every non-negative function $h$ that is convex over $[-a,b]$, $a,b \geq 0$, whose minimum is at $0$, the following inequality 
\begin{equation}
      \label{eq:convex_inequality_3}
      2\left(\frac{b+a}{2}\right)^2+h(-a)^2+h(b)^2 \geq 2 \int_{-a}^b h(x)\sqrt{1+(h^\prime(x))^2}dx.\end{equation}
    We denote $\ell=h(0)$. On the right-hand-side, we split the integral over $[-a,0]$ and $[0,b]$; we first estimate the part over $[0,b]$. By observing that $$\sqrt{1+(h^\prime)^2}=h^\prime + \frac{1}{h^\prime + \sqrt{1+(h^\prime)^2}},$$
we can write
$$I_b:=2\int_0^b h(t)\sqrt{1+(h^\prime(t))^2}dt=2\int_0^b h(t)h^\prime(t) dt + 2\int_0^b\frac{h(t)}{h^\prime(t) + \sqrt{1+(h^\prime(t))^2}}dt.$$
The first summand  is equal to  
$h^2(b)-\ell^2.$
In the second integral, we use that $h(t)=\ell+\int_0^th^\prime(s)ds$. We therefore obtain

\begin{align}
I_b=&h^2(b)-\ell^2 + \int_0^b\frac{2\ell dt}{h^\prime(t) + \sqrt{1+(h^\prime(t))^2}}+2 \int_{0}^b\int_0^t\frac{h^{\prime}(s) d s d t}{h^{\prime}(t)+\sqrt{1+h^{\prime}(t)^2}}\nonumber\\
=& h^2(b)-\ell^2 + \int_0^b\frac{2\ell dt}{h^\prime(t) + \sqrt{1+(h^\prime(t))^2}}+2 \int_{0}^b\int_s^b\frac{h^{\prime}(s) d t d s}{h^{\prime}(t)+\sqrt{1+h^{\prime}(t)^2}}\nonumber\\
\le & 
h^2(b)-\ell^2 + \int_0^b\frac{2\ell dt}{h^\prime(t) + \sqrt{1+(h^\prime(t))^2}}+2 \int_0^b \frac{h^{\prime}(s)(b-s) d s}{h^{\prime}(s)+\sqrt{1+h^{\prime}(s)^2}}, \label{deriv}
\end{align}
where in the last inequality we used  $h^\prime(s) \leq h^\prime(t)$, for $s\le t$. 


Next we note that the functions $z \mapsto \frac{z}{z+\sqrt{1+z^2}}$  and $h^\prime(s)$ are increasing. Thus,
$$s \mapsto \frac{h^{\prime}(s)}{h^{\prime}(s)+\sqrt{1+h^{\prime}(s)^2}}$$
is increasing on $[0,b]$. However, the function $b-s$ is decreasing on $[0,b]$. Since the average of the product of two non-negative functions with opposite monotonicity does not exceed the product of their averages, one obtains
$$2 \int_0^b \frac{h^{\prime}(s)(b-s) d s}{h^{\prime}(s)+\sqrt{1+h^{\prime}(s)^2}} \leq \int_0^b \frac{bh^\prime(s) ds}{h^\prime(s)+\sqrt{1+h^\prime(s)^2}}.$$
Combining these facts, we obtain
$$I_b \leq  h^2(b)-\ell^2 + \int_0^b\frac{(2\ell +bh^\prime(t)) dt}{h^\prime(t) + \sqrt{1+(h^\prime(t))^2}} = h^2(b)-\ell^2 +b \frac{2\ell +bg }{g + \sqrt{1+g^2}},$$
where $g=h^\prime(s^\star)$ for some $s^\star \in [0,b]$, and the equality holds via the mean value theorem for integrals.
By defining 
$$
I_a:=2\int_{-a}^{0} h(t)\sqrt{1+(h^\prime(t))^2}dt=2\int_{0}^{a} h(-t)\sqrt{1+(h^\prime(-t))^2}dt,$$ the above procedure can be repeated to obtain
$$I_a \leq  h^2(-a)-\ell^2 +a \frac{2\ell +aq }{q + \sqrt{1+q^2}},$$
where $q=-h^\prime(-\bar s)$ for some $-\bar s \in [-a,0]$. Combining the estimates together, we obtain
$$ 2 \int_{-a}^b h(x)\sqrt{1+(h^\prime(x))^2}dx \leq h^2(-a) +h^2(b) -2\ell^2 + a \frac{2\ell +aq }{q + \sqrt{1+q^2}} + b \frac{2\ell +bg }{g + \sqrt{1+g^2}}.$$
Denoting $c=\frac{a+b}{2}$, all that remains to show is that, for any $\ell, q,g \ge 0$,
\begin{equation}a \frac{2\ell +aq }{q + \sqrt{1+q^2}} + b \frac{2\ell +bg }{g + \sqrt{1+g^2}} \leq 2c^2 + 2\ell^2.
\label{eq:almost_Naz}
\end{equation}
Letting $\lambda^\star = \frac{b-a}{2}$, we can write $a=c-\lambda^\star$ and $b=c+\lambda^\star$. Consider the function
$$w(\lambda)= (c-\lambda) \frac{2\ell +(c-\lambda)q }{q + \sqrt{1+q^2}} + (c+\lambda) \frac{2\ell +(c+\lambda)g }{g + \sqrt{1+g^2}},$$
on the interval $[0,c]$. 
It thus suffices to show
$$\max_{\lambda\in [0,c]}w(\lambda) \leq 2c^2 + 2\ell^2.$$
Notice that $w(\lambda)$ is a quadratic function in $\lambda$ and the coefficient of $\lambda^2$ is positive. Thus, its maximum over $[0,c]$ occurs at an endpoint. Let us first check $\lambda =c$. We want to show
\begin{equation}\label{fed1}
\frac{c(2 \ell+2 c g)}{g+\sqrt{1+g^2}}\leq c^2+\ell^2,\end{equation}
i.e.
\begin{equation}\label{agm}2 c \ell \leq\left(\sqrt{1+g^2}-g\right) c^2+\left(\sqrt{1+g^2}+g\right) \ell^2,\end{equation}
which follows from the observation that
$$\left(\sqrt{1+g^2}-g\right)\left(\sqrt{1+g^2}+g\right)=1$$
and the inequality for arithmetic and geometric means. The case $\lambda =0$ becomes
$$
c \frac{2\ell + c q }{q + \sqrt{1+q^2}} + c \frac{2\ell +c g }{g + \sqrt{1+g^2}} \le 2(c^2+\ell^2),$$
and we finish by using (\ref{fed1}) twice.

To study the equality cases of (\ref{eq:convex_inequality_3}), we notice that it follows from (\ref{deriv}) that $h'(t)$ must be a constant. Thus $h(t)=\alpha t + \beta$. Next we must have equality in (\ref{eq:almost_Naz}), with $q=g=\alpha$ and thus equality in (\ref{agm}). Thus $c=(\alpha + \sqrt{1+\alpha^2}) \ell.$ First consider $\alpha \ge 0$ and assume  $a=0$ (we will shift, $-a,b$ and $h$ accordingly). If $\alpha \ge 0$ then
$$
\frac{b}{2}
=(\alpha + \sqrt{1+\alpha^2}) \beta \mbox{  
or 
}
\beta=\frac{\sqrt{1+\alpha^2}-\alpha}{2} b.
$$
Thus, shifting the function back we get, for the case $\alpha \ge 0$
$$
h(t) = \alpha t + \frac{\sqrt{1+\alpha^2}-\alpha}{2} b + \frac{\sqrt{1+\alpha^2}+\alpha}{2} a,
$$
as claimed (recall at the beginning of the proof, we replaced $a$ with $-a$). The case $\alpha \le 0$ follows by considering the function $h_1(t)=h(-t)=-\alpha t +\beta,$ where $t\in [-b,a]$.
\end{proof}

\subsection{Optimized form of the inequality}
In this section, we optimize the inequality \eqref{eq:convex_inequality_2}; first, consider the case when $b=1$ and $a=0$. Then, we obtain, for a non-negative function $h$ convex over $[0,1]$,
\begin{equation}
      \label{eq:convex_inequality_4}
      \frac{1}{4}+\frac{h(0)^2+h(1)^2}{2} \geq \int_0^1 h(x)\sqrt{1+(h^\prime(x))^2}dx.\end{equation}
In this case, there is equality only for functions of the form $h_{\text{lin}}(x)=(h(1)-h(0))x +h(0)$ under the additional constraint that $h(0),h(1)$ are two positive numbers satisfying $4h(1)h(0)=1$. But actually, inequality \eqref{eq:convex_inequality_4} is equivalent to \eqref{eq:convex_inequality_2}. Indeed, if $\tilde h$ is non-negative and convex over $[a,b]$, set $h(x)=(b-a)^{-1}\tilde h ((b-a)x+a)$ in \eqref{eq:convex_inequality_4} to obtain \eqref{eq:convex_inequality_2} for $\tilde h$.   

      We re-introduce the notation $L(h)=\text{arc-length of $h$}=\int_0^1 \sqrt{1+(h^\prime(x))^2}dx.$ Consider not $h$ but $h-c$ for some $c\in \R$. Then, \eqref{eq:convex_inequality_4} becomes
      \begin{equation}
      \label{eq:convex_inequality_5}
      c^2 -c(h(1)+h(0))+ cL(h)+\frac{1}{4}+\frac{h(0)^2+h(1)^2}{2} \geq \int_0^1 h(x)\sqrt{1+(h^\prime(x))^2}dx.
      \end{equation}
      We can then take infimum over $c$ in \eqref{eq:convex_inequality_5}. Letting the left-hand-side be a function in $c$, taking the derivative and solving yields 
      \begin{equation}
      \label{eq:opt}
      c_{\text{opt}}=\frac{h(1)+h(0)}{2}-\frac{L(h)}{2}.\end{equation}
      As a quick aside, if we denote $h_{\text{opt}}=h-c_{\text{opt}}$, then $h_{\text{opt}}(0)=\frac{h(0)-h(1)}{2}+\frac{L(h)}{2}$ and similarly for  $h_{\text{opt}}(1)$. Thus, \begin{equation}
          h_{\text{opt}}(1)+h_{\text{opt}}(0)=L(h)=L(h_{\text{opt}})
          \label{eq:opt_func}
      \end{equation}(from the translation invariance of arc-length). In fact,
      $h=h_{\text{opt}}$ (i.e. $c_{\text{opt}}=0$) if and only if $h(1)+h(0)=L(h)$. We deduce that $h_{\text{opt}}$ is the unique translate of $h$ satisfying \eqref{eq:opt_func}.
      
      Anyway, inserting $c_{\text{opt}}$ from \eqref{eq:opt} into \eqref{eq:convex_inequality_5} (and multiplying by a negative sign) yields
      \begin{equation}
      \label{eq:convex_inequality_6}
\int_0^1 \left[\frac{h(0)+h(1)}{2}-h(x)\right]\sqrt{1+(h^\prime(x))^2}dx \geq \frac{1}{4}\left[L(h)^2-((h(1)-h(0))^2+1)\right].
      \end{equation}
      We emphasize that \eqref{eq:convex_inequality_6} is, by construction, \eqref{eq:convex_inequality_4} applied to $h_{\text{opt}}$ but written in terms of $h$, and \eqref{eq:convex_inequality_6} holds without any further assumptions on $h$ besides being non-negative and convex over $[0,1]$. However, observe from the convexity of the function $h$ that
      $$L(h)^2\geq L(h_{\text{lin}})^2= (h(1)-h(0))^2+1,$$
      where $h_{\text{lin}}$ is the linear interpolation between $h(0)$ and $h(1)$.
      Hence, we can write Proposition~\ref{prop:fedja-cvx} in the following easier to visualize form.

\begin{prop}
Let $h:[0,1]\rightarrow \R$ be convex, and $h_{\text{lin}}$ be the linear interpolation between $h(0)$ and $h(1)$. Then:
\begin{equation}
\label{eq:convex_inequality_7}
\int_0^1 \left[\frac{h(0)+h(1)}{2}-h(x)\right]\sqrt{1+(h^\prime(x))^2}dx \geq \frac{1}{4}\left[L(h)^2- L(h_{\text{lin}})^2\right],
\end{equation}
where $L(\cdot)$ denotes arc length.
\end{prop}

Note that the integral on the left hand side may be considered as an integral with respect to the ``arc length measure'' on $[0,1]$ induced by $h$.

\subsection{Open question: higher dimensions}
In this section, we leave open a higher-dimension analogue of Theorem~\ref{t:contr}. That is, we leave open an inequality that is equivalent to
$\mu^+(\partial(K+L)) \geq \mu^+(\partial K)$ for $L$ a zonoid containing the origin and $\mu$ the measure on $\R^n$ with density $|x|^2$. Firstly, repeat the parameterization of $K$ as $\{(\bar x,t)\in\R^{n-1}\times\R:g(\bar x)\leq t \leq f(\bar x)\}$ where $g$ is some convex function, and $f$ is some concave function. We now suppose that $L=[0,e_n]$. If the inequality was true, we would have  that $\mu(K;[0,e_n],B_2^n) \geq 0$. We could calculate the formula like in the $n=2$ case, but this time we use \eqref{eq:sym_seg_0} directly to obtain that we would have 
$$\int_{\partial C}(|y|^2+f(y)^2)dy + \frac{2}{n-1}\int_C f(z)\sqrt{1+|\nabla f(z)|^2}dz\geq 0$$
for every function $f$ that is concave on a compact, convex set $C\subset \R^{n-1}$. Again notice the inequality is true if and only if the inequality is true for $f_-$. We then set $h=-f$ and get that the higher-dimensional analogue of Theorem~\ref{t:contr} is true if and only if the following question holds.
\begin{ques}
    Fix $n\geq 3$ and let $h$ be a non-negative, convex function on $\R^{n-1}$. Then, is it true that for every compact, convex subset $C$ of the support of $h$, one has
    \begin{equation}\int_{\partial C}(|y|^2+h(y)^2)d\mathcal{H}^{n-2}(y) \geq \frac{2}{n-1}\int_C h(z)\sqrt{1+|\nabla h(z)|^2}dz?
    \label{eq:naz_higher}
    \end{equation}
\end{ques}

We remark that we verified the inequality \eqref{eq:naz_higher} for $n=3$ with $C=[0,1]\times [0,\epsilon]$ and $h=x_1^\alpha+\lambda x_2^\beta$, $\alpha,\beta\in\{1,2,3\}$, $\epsilon,\lambda >0$.

      \section{Appendix C: supermodular measures must have density}
      \label{sec:density}
      Inspired by Borell's classification of $s$-concave measures \cite{Bor75a}, we show that a Radon measure being supermodular over the class of convex bodies implies the measure has density. In particular, one can then approximate the density with continuously differentiable densities, each of which is supermodular, yielding alternative starting points in the proofs of Theorems~\ref{t:Radon} and \ref{t:surarecom}.
\begin{lem}
\label{l:density}
    Let $\mu$ be a Radon measure on $\R^n$ such for every pair of convex bodies $A$ and $C$ and every $r\geq 0$, one has
    \begin{equation}
        \mu(A+rB_2^n+C) +\mu(A) \geq \mu(A+rB_2^n) + \mu(A+C).
        \label{eq:sup_mod_ball}
    \end{equation}
    Then $\mu$ is absolutely continuous with respect to the Lebesgue measure. 
\end{lem}
\begin{proof}
It is easy to see that \eqref{eq:sup_mod_ball} cannot hold if the singular part of $\mu$ contains a Dirac point mass. Indeed, let $x$ be a point mass so that $\mu(\{x\})=a>0$. Let $\{x_k\}$ be a sequence such that $x_n \to x$ in norm and $\mu(\{x_k\})=0$ for all $k\in \N$; such a sequence exists, otherwise we would contradict the Radon assumption. We use \eqref{eq:sup_mod_ball} with $A=\{x_k\},$ and let $C=rB_2^n$, to get
$$\mu(x_k +2rB_2^n) \geq 2 \mu(x_k+rB_2^n).$$
Now select  a sequence of radii $r=r_k \to 0$ such that $x_k+r_kB_2^n \supset x_{k+1}+r_{k+1} B_2^n$. 
Then, the right-hand side of the above inequality converges to  $2a$ and the left-hand side converges to $a$, which gives a contradiction.




We will use the notation $Q(x,r)$ for the cube centered at $x\in \R^n$ and whose side-lengths is  $r$. We will need the following classical fact (see e.g. \cite[Theorem 8.11]{Rud87:book}): if $\mu$ is a Borel measure which is not absolutely continuous with respect to the Lebesgue measure,
then there exists a point $z$ and a sequence of integers 
$p_j \rightarrow \infty$, such that
$$
\lim_{j \rightarrow \infty}\, 2^{np_j}\, \mu\big(Q(z,2^{-p_j})\big) = \infty.
$$
From the fact that $\mu$ is a Borel measure, we may replace $Q(z,r)$ with $B_2^n(z,r)$, the ball of radius $r$ centered at $z$, by using set-inclusions and making $p_j$ bigger, i.e. one has that there exists a sequence of integers 
$p_j \rightarrow \infty$, such that
$$
\lim_{j \rightarrow \infty}\, 2^{np_j}\, \mu\big(B_2^n(z,2^{-p_j})\big) = \infty.
$$

\noindent Let us call any point $z$ that has the above property a \textit{bad point.} Fix $x \in \R^n$, and notice that \eqref{eq:sup_mod_ball}, with $A=\{z\}$, $B=B(0,2^{-p_j})$ and $C=\{x-z\}$,  yields 
$$\mu(B_2^n(x,2^{-p_j})) + \mu(\{z\})\geq \mu(B_2^n(z,2^{-p_j})) + \mu(\{x\}) \geq \mu(B_2^n(z,2^{-p_j})).$$ 

Using that $\mu(\{z\})=0$, one then deduces that every $x\in\R^n$ is a bad point since
$$2^{np_j}\mu(B_2^n(x,2^{-p_j}))\geq 2^{np_j}\mu(B_2^n(z,2^{-p_j})):=C_j,$$
where $C_j\to\infty$ as $j\to\infty$. Now, fix an arbitrary cube $Q$. Notice, for a fixed $j$ large enough, that $Q$ can be subdivided (with controllable error) into disjoint cubes of the form $Q_i=Q(x_i,2^{-(p_j-1)})$. Thus,
$$\mu(Q)\geq \sum_{i}\mu(Q_i) \geq 2^{-2n}C_j\sum_{i}\vol_n(Q_i) \geq 2^{-2n}C_j\left(\vol_n(Q)-n2^{-(p_j-2)}\vol_n(Q)^\frac{n-1}{n}\right).$$
Sending $j\to\infty$ yields that $\mu(Q)=\infty$ for every cube $Q$, contradicting $\mu$ being a Radon measure. We deduce that the support of the singular part of $\mu$ is the empty set, and hence $\mu$ is absolutely continuous with respect to the Lebesgue measure.
\end{proof}

\printbibliography


\noindent Matthieu Fradelizi
\\
Univ Gustave Eiffel, Univ Paris Est Creteil, CNRS LAMA UMR8050, F-77447 Marne-la-Vall\'ee, France.
\\
E-mail address: matthieu.fradelizi@univ-eiffel.fr
\vspace{2mm}
\\
\noindent Dylan Langharst
\\
Institut de Math\'ematiques de Jussieu, Sorbonne Universit\'e, 75005 Paris, France.
\\
E-mail address: dylan.langharst@imj-prg.fr
\vspace{2mm}
\\
\noindent Mokshay Madiman 
\\
University of Delaware, Department of Mathematical Sciences, 501 Ewing Hall, Newark, DE 19716, USA. 
\\
E-mail address: madiman@udel.edu
\vspace{2mm}
\\
\noindent Artem Zvavitch
\\
Department of Mathematical Sciences, Kent State University, Kent, OH 44242, USA. 
\\
E-mail address: azvavitc@kent.edu

\end{document}